\documentclass[bj]{imsart}
\usepackage{xr-hyper}
\usepackage{amsthm,amsmath}
\RequirePackage[OT1]{fontenc}
\RequirePackage[colorlinks,citecolor=blue,urlcolor=blue]{hyperref}

          \usepackage{amssymb}
          \usepackage{amsfonts}
          \usepackage[english]{babel}
          \usepackage[utf8]{inputenc}
          \usepackage{enumerate}
          \usepackage{graphicx}
          \usepackage{url}
          \usepackage{color}

\usepackage[numbers,sort&compress]{natbib}

\arxiv{arXiv:1712.08058}

\startlocaldefs

\newtheorem{Thm}{Theorem}[section]
\newtheorem{Lem}[Thm]{Lemma}
\newtheorem{Prop}[Thm]{Proposition}

\theoremstyle{remark}
\newtheorem{Rem}[Thm]{Remark}

\numberwithin{equation}{section}

\newcommand{\comment}[1]{}
\newcommand{\ind}{{\bf 1}}
\def\indd#1{{\ind}_{\{#1\}}}
\def\inddd#1{{\ind}_{\left\{#1\right\}}}
\def\indn#1{\{#1_n\}_{n\in\N}}
\newcommand{\proba}{\mathbb P}
\newcommand{\esp}{{\mathbb E}}

\newcommand{\inv}{^{-1}}
\newcommand{\cov}{{\rm{Cov}}}
\newcommand{\var}{{\rm{Var}}}

\def\B{{\mathbb B}}

\def\G{{\mathbb G}}

\def\Z{{\mathbb Z}}

\newcommand{\eqnh}{\begin{eqnarray*}}
\newcommand{\eqne}{\end{eqnarray*}}
\newcommand{\eqnhn}{\begin{eqnarray}}
\newcommand{\eqnen}{\end{eqnarray}}
\newcommand{\equh}{\begin{equation}}
\newcommand{\eque}{\end{equation}}

\def\summ#1#2#3{\sum_{#1 = #2}^{#3}}
\def\prodd#1#2#3{\prod_{#1 = #2}^{#3}}

\newcommand{\eqd}{\stackrel{\rm d}{=}}

\def\topp#1{^{(#1)}}

\def\nn#1{{\left\|#1\right\|}}

\def\abs#1{\left|#1\right|}

\def\ccbb#1{\left\{#1\right\}}
\def\sccbb#1{\{#1\}}
\def\pp#1{\left(#1\right)}
\def\spp#1{(#1)}
\def\bb#1{\left[#1\right]}

\def\mmid{\;\middle\vert\;}

\def\ip#1{\left\langle#1\right\rangle}

\def\floor#1{\left\lfloor #1 \right\rfloor}
\def\sfloor#1{\lfloor #1 \rfloor}

\def\vv#1{{\boldsymbol #1}}

\def\vvi{{\boldsymbol i}}
\def\vvj{{\boldsymbol j}}

\def\vvl{{\boldsymbol \ell}}

\def\vvn{{\boldsymbol n}}

\def\vvs{{\boldsymbol s}}
\def\vvt{{\boldsymbol t}}
\def\vvu{{\boldsymbol u}}

\def\vvw{{\boldsymbol w}}
\def\vvH{{\boldsymbol H}}

\def\qmand{\quad\mbox{ and }\quad}

\def\mwith{\mbox{ with }}
\def\qmwith{\quad\mbox{ with }\quad}
\def\mfa{\mbox{ for all }}

\def\mmas{\mbox{ as }}

\def\wt#1{\widetilde{#1}}
\def\wb#1{\overline{#1}}
\def\what#1{\widehat{#1}}

\def\weakto{\Rightarrow}

\def\limn{\lim_{n\to\infty}}

\def\Z{{\mathbb Z}}

\def\R{{\mathbb R}}

\def\N{{\mathbb N}}

\endlocaldefs

\begin{document}

\begin{frontmatter}

% "Title of the Paper"
%\title{???}

%\runtitle{???}

\title{Operator-scaling Gaussian random fields via aggregation}

\runtitle{Operator-scaling Gaussian random fields via aggregation}

\begin{aug}
% indicate corresponding author with \corref{}
% \author{\fnms{John} \snm{Smith}\thanksref{a}\corref{}\ead[label=e1]{smith@foo.com}\ead[label=e2,url]{www.foo.com}}
% \address[a]{\printead{e1};\printead{e2}}

\author{\fnms{Yi} \snm{Shen}\thanksref{a}\ead[label=e1]{yi.shen@uwaterloo.ca}}
\and
\author{\fnms{Yizao} \snm{Wang}\thanksref{b}\corref{}\ead[label=e2]{yizao.wang@uc.edu}}
\address[a]{Department of Statistics and Actuarial Science,
University of Waterloo,
Mathematics 3 Building,
200 University Avenue West
Waterloo, Ontario N2L 3G1,
Canada.
\printead{e1}}
\address[b]{
Department of Mathematical Sciences,
University of Cincinnati,
2815 Commons Way, ML--0025,
Cincinnati, OH, 45221-0025, USA.
\printead{e2}}

\runauthor{Shen and Wang}

\affiliation{University of Waterloo and University of Cincinnati}

\end{aug}

\begin{abstract}\
We propose an aggregated random-field model, and investigate the scaling limits of the aggregated partial-sum random fields. In this model, each copy in the aggregation is a $\pm1$-valued random field built from two correlated one-dimensional random walks, the law of each determined by a random persistence parameter. A flexible joint distribution of the two parameters is introduced, and given the parameters the two correlated random walks are conditionally independent. For the aggregated random field, when the persistence parameters are independent, the scaling limit is a fractional Brownian sheet. When the persistence parameters are
tail-dependent, characterized in the framework of multivariate regular variation,
 the scaling limit is more delicate, and in particular depends on the growth rates of the underlying rectangular region along two directions: at different rates different operator-scaling Gaussian random fields appear as the region area tends to infinity. In particular, at the so-called critical speed,  a large family of Gaussian random fields with long-range dependence arise in the limit. We also identify four different regimes at non-critical speed where fractional Brownian sheets arise in the limit.
\end{abstract}


% history:
% \received{\smonth{1} \syear{0000}}

%\tableofcontents

\end{frontmatter}

\section{Introduction and main results}
Long-range dependence phenomena are well known in various areas of applications, including notably econometrics, finance, and network traffic modeling. It is also referred to as long memory, particularly in time-series setup. Traditionally, a stationary stochastic process with finite second moment is considered to have long-range dependence if, roughly speaking, either the covariance function has the power law decay, or the spectral density has a singularity at the origin. The two approaches are referred to in the literature as the %YZ1206 time domain
time-domain
 approach and the %YZ1206 frequency domain
frequency-domain
 approach, respectively. Recently, interpretations of long-range dependence in terms of limit theorems have become more and more popular: a stochastic model of interest may be viewed to have long-range dependence if,
for example,
when compared to a similar model with short-range dependence, the normalization in certain limit theorem for partial sums   is of a different order. The other model in comparison here may be the same model but with a different choice of parameter, or a much simplified model for which the short-range dependence has been well understood. The anomalous normalization already indicates the qualitatively different behavior of the model.
    Moreover, a functional limit theorem provides a %YZ1206 much
    more
     precise description of the macroscopic dependence, in terms of the limit process, and new families of stochastic processes have been discovered in this way.
  At the same time,
  limit theorems also provide insightful explanation on how long-range dependence appears, and the limit process,    due to the intriguing dependence structure inherited from the discrete model, may be of independent interest for further investigation.  Excellent references on long-range dependence in stochastic processes and applications include for example \citep{pipiras17long,beran13long,samorodnitsky16stochastic}.

In the investigation of long-range dependence, two classes of models have prominent roles: models via aggregation and models via filtration (often in the form of fractionally integrated processes or random fields). We shall focus on aggregated models in this paper. One of the most famous aggregated models is due to \citet{robinson78statistical} and \citet{granger80long}, who showed that the aggregation of autoregressive processes with random parameters may lead to long memory. In particular, this model has received huge success in explaining the long memory phenomena in many economics and financial data sets in the econometrics literature. Another area where aggregated models have been extensively investigated is modeling long memory in network traffic.  See for example \citep{mikosch02network,mikosch07scaling,kaj08convergence}.

In the spatial setup, however, aggregated random fields have been much less developed than their one-dimensional counterparts. See for example \citet{lavancier06long,lavancier11aggregation} and references therein.
In particular, we are interested in aggregated spatial models of which, if scaled appropriately, the limit random fields are {\em operator-scaling} Gaussian random fields.
We say a random field $\{\G_\vvt\}_{\vvt\ge \vv0}$ is operator-scaling, if for some $\beta_1,\beta_2,H>0$ we have
\equh\label{eq:OS}
\ccbb{\G_{\lambda^{\beta_1}t_1,\lambda^{\beta_2}t_2}}_{\vvt\ge \vv0} \eqd \lambda^H\ccbb{\G_\vvt}_{\vvt\ge \vv0}, \mfa \lambda>0.
\eque
This is actually a special case of the operator-scaling property introduced by \citet{bierme07operator}. This property is  an extension of the %YZ1206 well known
well-known
 self-similar property for one-dimensional stochastic processes.
Most operator-scaling random fields are anisotropic in the sense that they have different scaling properties in different directions, a very desirable property from modeling point of view.
At the same time, this property also makes the analysis of such Gaussian random fields very challenging, and they have attracted much research interest since its introduction.
See for example
\citep{xiao09sample,xiao13recent,meerschaert13fernique,li15exact} for recent developments on path properties of operator-scaling Gaussian random fields.
Most operator-scaling random fields, as their one-dimensional counterparts, exhibit long-range dependence. Gaussian random fields with long-range dependence are known to have applications in medical image processing \citep{bierme09anisotropic,lopes09fractal} and hydrology \citep{benson06aquifer,meerschaert13hydraulic}. Econometric interpretation for aggregated models has also been discussed in the literature \citep{leonenko13disaggregation}.
In terms of limit theorems,
not many models that scale to anisotropic operator-scaling Gaussian random fields have been known, including notably \citep{lavancier07invariance,bierme14invariance,wang14invariance,puplinskaite15scaling,puplinskaite16aggregation,pilipauskaite17scaling,bierme17invariance,lavancier11aggregation,durieu19random}. Among these, only \citet{lavancier11aggregation} and \citet{puplinskaite16aggregation} considered certain aggregated random fields (not strictly in our sense though, see Remark \ref{rem:double_limit}), while only \citet{puplinskaite16aggregation} considered anisotropic aggregated ones.

In this paper, we propose a new aggregated random-field model that scales to a large family of anisotropic operator-scaling Gaussian random fields.  Our model may be viewed as an extension of the approximation of fractional Brownian motions by aggregated random walks introduced by \citet{enriquez04simple}. In fact, \citet{enriquez04simple} proposed two different models for approximation of fractional Brownian motions with Hurst index $H>1/2$ and $H<1/2$ respectively, and our extension is based on the one for $H>1/2$ here. In this case, the Enriquez model can be viewed as an aggregation of independent copies of correlated random walks, where the law of each correlated random walk is completely determined by a random persistence parameter.
 We will investigate  in another paper the extension of the other Enriquez model (for $H<1/2$), which is of a different nature.

 In particular, our model inherited a prominent feature from the one-dimensional model that, in the aggregated model, each independent copy of the random field (or stochastic process in one dimension) takes only $\pm1$-values. It is appealing to restrict the values of model to $\pm1$ from numerical simulation point of view. It also provides better insight on the dependence structure. Besides \citep{enriquez04simple}, a few recent limit theorems for $\pm1$-valued discrete models with long-range dependence include \citep{hammond13power,durieu16infinite,bierme17invariance,durieu19random}.

The extension to random fields, however, is by no means straightforward. For each random field in the aggregation we are now searching for $\pm1$-valued models with non-trivial anisotropic dependence. The key idea is to consider two independent one-dimensional random walks as in the Enriquez model, and define the random field as the product
 of the two sequences of $\pm1$-valued steps of each; the dependence of the so-obtained random field is then determined by assigning an %YZ1206: here and the following sentence appropriate joint distribution
 appropriate tail-dependence structure
 of the two persistence parameters (and keeping the random walks conditionally independent). Our %choice
 modeling of the %joint distribution
 tail dependence is flexible, so that a large family of random fields arise in the limit, and also computable, so that we have explicit form of the asymptotic covariance of the limit Gaussian field, which is in general much more complex than in one dimension.

Below, we first review the Enriquez model in dimension one, and then introduce our generalization. The main results are then presented in Section~\ref{sec:main}.
\subsection{Enriquez model in dimension one}
\citet{enriquez04simple} proposed two aggregated models that scale to  fractional Brownian motions, with Hurst index $H\in(1/2,1)$ and $H\in(0,1/2)$ respectively. We shall focus exclusively on the first one and its generalization to random fields, and we refer to this one as the Enriquez model in this paper, for the sake of simplicity.

The Enriquez model consists of aggregation of a family of independent  $\{\pm1\}$-valued stationary sequences, with a parameter $H\in(1/2,1)$. The model is as follows.
First, a random variable $q$ is sampled from the probability distribution $\mu_H$ on $(0,1)$ defined as
\equh\label{eq:mu_H}
\mu_H(dq) = (1-H)2^{3-2H}(1-q)^{1-2H}\indd{q\in(1/2,1)}dq.
\eque
 For the sake of convenience, with a slight abuse of notation we let $q$ denote both a random variable in general and the variable in the density formula. Then, a sequence of random variables $\indn\varepsilon$ is sampled iteratively: $\varepsilon_1$ is a $\{\pm1\}$-valued symmetric random variable, and for each $n\in\N$  given the past and $q$, $\varepsilon_{n+1}$ is set to take the same value of $\varepsilon_n$ with probability $q$, and the opposite with probability $1-q$. The law of the so-sampled sequence $\indn\varepsilon$ is determined by, given
$q$ and $\varepsilon_1$,
\equh\label{eq:enriquez}
\proba(\varepsilon_{n+1} = 1\mid \varepsilon_1,\dots,\varepsilon_n,q) = q\indd{\varepsilon_n = 1} + (1-q)\indd{\varepsilon_n = -1}, n\in\N.
\eque
Let $S_n:=\varepsilon_1+\cdots+\varepsilon_n$ denote the partial sum of the stationary sequence.
For each $q$ fixed, the sequence $\indn S$ can be viewed as a correlated $\{\pm1\}$-valued random walk, and $q$ is referred to as the persistence of the random walk. The partial-sum process $\{S_n(t)\}_{t\in[0,1]}$ of this sequence is denoted by
\[
S_n(t):=\summ j1{\floor{nt}}\varepsilon_j, t\in[0,1], n\in\N.
\]
Next, consider  i.i.d.~copies of the stationary sequence $\varepsilon$, each copy denoted by $\varepsilon^i\equiv\{\varepsilon^i_n\}_{n\in\N}$. Let $\{S_n^i(t)\}_{t\in[0,1]}$ denote the partial-sum processes of the $i$-th sequence.  Let $\{m(n)\}_{n\in\N}$ denote a increasing sequence of integers, and   $\what S_n(t)$ denote the aggregated partial-sum process of $m(n)$ i.i.d.~sequences
\[
\what S_n(t) := \summ i1{m(n)}S_n^i(t) =  \summ i1{m(n)}\summ j1{\floor{nt}}\varepsilon_j^i.
\]
\citet[Corollary 1]{enriquez04simple} proved that if $\limn m(n)/n^{2-2H} = \infty$, then
\[
\ccbb{\frac{\what S_{n}(t)}{n^H}}_{t\in[0,1]}\weakto \sqrt{\frac{\Gamma(3-2H)}{H(2H-1)}}\ccbb{\B^H_t}_{t\in[0,1]}
\]
in $D([0,1])$, where $\B^H$ is the fractional Brownian motion, a centered Gaussian process with covariance function
\[
\cov\pp{\B^{H}_s,\B^{H}_t} = \frac{1}{2}(s^{2H}+t^{2H}-|s-t|^{2H}), s, t\geq 0.
\]

\subsection{An aggregated random-field model}
We consider the following generalization of the Enriquez model,  consisting of independent copies of a $\{\pm1\}$-valued stationary random field $\{X_\vvn\}_{\vvn\in\N^2}$. For each copy, a random vector %YZ0506 added $(q_1,q_2)$
$\vv q = (q_1,q_2)$
 is first sampled from a certain distribution $\mu$ on $[1/2,1)^2$ to be described later.
Next, given $q_1,q_2\in[1/2,1)$,  let $\varepsilon\topp k \equiv\indn{\varepsilon\topp k}, k=1,2$ be two conditionally independent
one-dimensional random walks with persistence $q_1$ and $q_2$ respectively as in the original Enriquez model (each starting with $\proba(\varepsilon_1 \topp k = \pm1) = 1/2$ and following the dynamics determined by~\eqref{eq:enriquez}).
 Then, consider the stationary random field
\[
X_\vvj := \varepsilon_{j_1}\topp1\varepsilon_{j_2}\topp2, \vvj\in\N^2.
\]
The stationarity of $X$ is easy to verify, regardless of the choice of $\mu$. Let
\[
S_\vvn(\vvt):=\sum_{\vvj\in[\vv1,\vvn\cdot\vvt]}X_\vvj
\]
denote the partial sum of the random field.
Here and below, $\vv n\cdot\vv t = (n_1t_1,n_2t_2)\in\R^2$ and $[\vv a,\vv b]$ is understood as
\[
[\vv a,\vv b] \equiv ([a_1,b_1]\times[a_2,b_2])\cap \Z^2, \vv a,\vv b\in\R^2
\] throughout the paper.

Next,  let $\{X^i\}_{i\in\N}$ be i.i.d.~copies of $X$, and define  $S_\vvn^i(\vvt)$ similarly as $S_\vvn(\vvt)$. We then consider the aggregated partial-sum random field $\{\what S_\vvn(\vvt)\}_{\vvt\in[0,1]^2}$ given by
\[
\what S_\vvn(\vvt):= \summ i1{m(\vvn)}S_\vvn^i(\vvt)\equiv \summ i1{m(\vvn)}\sum_{\vvj\in[\vv1,\vvn\cdot\vvt]}X_\vvj^{i},\vvt\in[0,1]^2,
\]
with $m(\vvn)\in\N$, the number of copies in the aggregation, to be chosen later.

Now we explain our choices of $\mu$, the law of %YZ0506 added $\vv q$
$\vv q = (q_1,q_2)$.
 Recall that this is a probability measure on $[1/2,1)^2$. We consider two cases of the model with drastically different behaviors:
\begin{enumerate}[(i)]
\item {\em independent persistence}, where we  assume that $q_1$ and $q_2$ are independent and with law $\mu_{H_1}$ and $\mu_{H_2}$, respectively. That is, $\mu = \mu_{H_1}\otimes \mu_{H_2}$.
This is the easiest case of our limit theorems.

\item {\em dependent persistence}, where we assume that $q_1$ and $q_2$ are {\em tail-dependent}
 in the specific way described below. This is the case to which most of our effort is devoted.
\end{enumerate}
In the case of dependent persistence, we introduce a specific and
 flexible model to characterize the tail dependence of $\vv q$ near $(1,1)$ as follows, which satisfies the
multivariate regular variation assumption  (see Remark \ref{rem:MRV} below).
To start with, and for the convenience of analysis later, we construct a random vector $\vv U\in(0,1]^2$ with law $\mu^*$, and set $\mu$ as its induced measure  on $[1/2,1)^2$ by
\[
%YZ0506 replaced 1 by (1,1)
\vv q = (1,1)-\frac{\vv U}{2}.
\]
To allow flexible and analytically tractable dependence between $U_1$ and $U_2$, let $R$ be a positive continuous random variable with probability density $r^{-2}dr$ over $(1,\infty)$, and $\vv W =(W_1,W_2)$ be a random vector taking values in
\[
\Delta_1 :=\{\vv w\in(0,1)^2:w_1+w_2=1\}
\] with law $\Lambda$. We assume that $R$ and $\vv W$ are independent, and let $\alpha_1, \alpha_2$ be two constants in $(0,2)$. Then introduce
\equh\label{eq:wtU}
\wt U_k:=(RW_k)^{-1/\alpha_k}, k=1,2,
\eque
and set
\equh\label{eq:U}
\vv U := \left\{
\begin{array}{ll}
(\wt U_1,\wt U_2)  & (\wt U_1,\wt U_2)\in(0,1)^2\\
(1,1) & \mbox{ otherwise}
\end{array}
\right.
\eque
to address the practical issue that $U_k$ by definition should be in $[0,1]$.
In this way, our aggregated random-field model with dependence parameters is completely determined by $\vv\alpha=(\alpha_1, \alpha_2)$ and $\Lambda$.
In the %YS1206 changed dependence
dependent persistence case, all these parameters have impact on the limit random fields
(see \eqref{eq:cov_G} below).
\begin{Rem}\label{rem:MRV}
It is natural to work in the framework of multivariate regular variation for $\vv q$, as it is clear that for non-trivial dependence structure, only the behavior of $\vv q$ near $(1,1)$ matters: as an extension of the one-dimensional model we need each $q_i$ to have power-law density near 1, and the new ingredient in two-dimensional modeling is to characterize the dependence of $\vv q$ at the tail $(1,1)$, a standard question in extreme value theory.
However, for modeling the tail dependence, traditionally in extreme value theory and also in our application, it is more convenient to work with multivariate regular variation assumption at either $(\infty,\infty)$ or $(0,0)$ \citep{resnick07heavy}.

More precisely, for our application the tail dependence of $\vv Z =R\vv W$ at $(\infty,\infty)$ plays a crucial role in the limit (see \eqref{eq:RW} and \eqref{eq:hnK}), which we model in the framework of multivariate regular variation in {\em polar coordinate}. A general assumption in this case should read as
\equh\label{eq:MRV}
n\proba\pp{\pp{\frac {\nn{\vv Z}}n, \frac{\vv Z}{\nn{\vv Z}}}\in\cdot}\stackrel{v}\rightarrow \frac{dr}{r^2}\times \Lambda(\cdot),
\eque
$\nn {\vv Z} := |Z_1|+|Z_2|$,
in the space of positive Radon measures on $[0,\infty]^2\setminus\{\vv0\}$ equipped with the vague topology, where $\stackrel v\rightarrow$ stands for the vague convergence, and $\Lambda$ is known as the {\em angular measure} that characterizes the tail dependence.
Our construction of $R\vv W$ in %YZ1206 \eqref{eq:U}
\eqref{eq:wtU}
is a well known procedure that implies \eqref{eq:MRV} \citep[Section 6.5.3]{resnick07heavy}.
The advantage of working with $R\vv W$ directly instead of the weaker assumption \eqref{eq:MRV} is to be able to obtain specific bounds quickly at various places, as the analysis is already quite involved.
\end{Rem}

\subsection{Main results}\label{sec:main}
Our main results are functional limit theorems on $\what S_\vvn(\vvt)$.
We first begin with the model with independent persistence.

\begin{Thm}\label{thm:1}
Consider the aggregated model with independent persistence ($\mu = \mu_{H_1}\otimes\mu_{H_2}$, $\mu_H$ as in \eqref{eq:mu_H} and $H_1,H_2\in(1/2,1)$). Assume also
\equh\label{eq:m(n)}
\lim_{\vvn\to\infty}\frac{n_1^{2-2H_1}n_2^{2-2H_2}}{m(\vvn)} = 0.
\eque
Then,
\[
\frac1{n_1^{H_1}n_2^{H_2}\sqrt {m(\vvn)}}\ccbb{\what S_\vvn(\vvt)}_{\vvt\in[0,1]^2}\weakto \sigma\ccbb{\B^{\vvH}_\vvt}_{\vvt\in[0,1]^2},
\]
in $D([0,1]^2)$ as $\vvn\to\infty$, where $\B^{\vvH}$ is a standard fractional Brownian sheet with covariance function
\[
\cov\pp{\B^{\vvH}_\vvs,\B^{\vvH}_\vvt} = \prodd k12\frac12\pp{s_k^{2H_k}+t_k^{2H_k}-|s_k-t_k|^{2H_k}}, \vvs,\vvt\ge 0,
\]
and
\[
\sigma := \prodd k12\pp{\frac{\Gamma(3-2H_k)}{H_k(2H_k-1)}}^{1/2}.
\]
\end{Thm}
Here and below, more precisely, we actually consider a sequence of vectors $\{\vvn(j)\}_{j\in\N}$ in $\N^2$ and the limit as $j\to\infty$. It is always assumed that $\lim_{j\to\infty}n_1(j) = \infty$ and $\lim_{j\to\infty}n_2(j) = \infty$, so that the partial sum is over a rectangular region of which the lengths of both directions tend to infinity.  For the sake of simplicity, throughout we drop the parameter $j$ and write $\vvn\to\infty$ instead of $j\to\infty$.
We will also write $a(\vvn)\sim b(\vvn)$ as $\vvn\to\infty$ if $\lim_{\vvn\to\infty}a(\vvn)/b(\vvn) = 1$.

For the model with dependent persistence, it turns out that the scaling limit depends on the relative growth rate of $n_1$ and $n_2$.
We first look at partial sums over rectangles increasing at the so-called {\em critical speed}:
\equh\label{eq:critical}
n_1^{\alpha_1}\sim n_2^{\alpha_2} \mmas \vvn\to\infty.
\eque
The following function
\equh\label{eq:Psi}
\Psi_{\vv\alpha,\Lambda}(\vv\theta):=\int_0^\infty\int_{\Delta_1}\prodd k12 \frac{2(rw_k)^{-1/\alpha_k}}{(rw_k)^{-2/\alpha_k}+\theta_k^2}\Lambda(d\vvw)r^{-2}dr
\eque
shows up in the harmonizable representation of the limit Gaussian random field.
The finiteness of $\Psi_{\vv\alpha,\Lambda}$ will be established in~\eqref{eq:hnK} below.

\begin{Thm}\label{thm:2}
Consider the aggregated model with dependent persistence and $\alpha_1,\alpha_2\in(0,2)$. If
%YZ1206 in addition,
\equh\label{eq:m(n)_I_AD}
\lim_{\vvn\to\infty}\frac{n_1^{\alpha_1}}{m(\vvn)} = 0,
\eque
then, at critical speed \eqref{eq:critical},
\[
\frac{n^{\alpha_1/2}}{|\vvn|\sqrt {m(\vvn)}}\ccbb{\what S_\vvn(\vvt)}_{\vvt\in[0,1]^2}\weakto \ccbb{\G^{\vv\alpha,\Lambda}_\vvt}_{\vvt\in[0,1]^2},
\]
in $D([0,1]^2)$ as $\vvn\to\infty$, where $|\vvn|=n_1n_2$, and $\G^{\vv\alpha,\Lambda}$ is a centered Gaussian random field with
\equh\label{eq:cov_G}
\cov(\G_\vvs^{\vv\alpha,\Lambda},\G_\vvt^{\vv\alpha,\Lambda}) = \frac1{(2\pi)^2}\int_{\R^2}\pp{\prodd k12\frac{(e^{is_k\theta_k}-1)\wb{(e^{it_k\theta_k}-1)}}{|\theta_k|^2}}\Psi_{\vv\alpha,\Lambda}(\vv\theta)d\vv\theta.
\eque
\end{Thm}
Next, when $\vvn$ does not grow at the critical speed \eqref{eq:critical}, we identify four different regimes. By symmetry, it suffices to assume
\equh\label{eq:symmetry}
n_1^{\alpha_1}\gg n_2^{\alpha_2},
\eque
by which we mean $\lim_{\vvn\to\infty} n_2^{\alpha_2}/n_1^{\alpha_1} = 0$. Under this assumption the following theorem identifies two regimes of non-critical speed,
and the other two regimes under the assumption $n_1^{\alpha_1}\ll n_2^{\alpha_2}$ can be read accordingly.
In the sequel we write
\[
\mathfrak c_H := B\pp{H-\frac12,\frac32-H}\frac\pi{H\Gamma(2H)\sin(H\pi)},
\]
where $B(a,b) = \int_0^1 x^{a-1}(1-x)^{b-1}dx$ is the Beta function.
\begin{Thm}\label{thm:non_critical}
Consider the aggregated model with dependent persistence and  $\alpha_1,\alpha_2\in(0,2)$. If  \[
\lim_{\vvn\to\infty}\frac{n_1^{2-2H_1}n_2^{2-2H_2}}{m(\vvn)} = 0,
\]
then at non-critical speed
\eqref{eq:symmetry},
\equh\label{eq:sub_regimes}
\frac1{n_1^{H_1}n_2^{H_2}\sqrt{m(\vvn)}}\ccbb{{\what S_\vvn(\vvt)}}_{\vvt\in[0,1]^2} \weakto \sigma\ccbb{\B_\vvt^{\vvH}}_{\vvt\in[0,1]^2}
\eque
in $D([0,1]^2)$ as $\vvn\to\infty$,
where $\B^\vvH$ is the fractional Brownian sheet with Hurst indices $\vvH$, for the following two cases depending on the value of $\alpha_1$. In each case, $\vvH$ and $\sigma^2$ are given accordingly:
\begin{enumerate}[(i)]
\item \label{case:D,beta<1}%$n_1^{\alpha_1}\gg n_2^{\alpha_2}$,
$\alpha_1>1$:%$\alpha'_2>\alpha_2, \alpha_1>1$,
\[
\displaystyle H_1 = \frac12, H_2 = 1-\frac{\alpha_2}2\pp{1-\frac1{\alpha_1}},\sigma^2 = 2\alpha_2\mathfrak c_{H_2}\int_{\Delta_1} w_1^{1/\alpha_1}w_2^{1-1/\alpha_1}\Lambda(d\vvw),
\]
\item \label{case:g,beta=1}%$n_1^{\alpha_1}\gg n_2^{\alpha_2}$,
$\alpha_1<1$:%$\alpha'_2>\alpha_2,\alpha_1<1$,
\[
\displaystyle H_1 = 1-\frac{\alpha_1}2, H_2 = 1, \sigma^2 = \alpha_1\mathfrak c_{H_1}\int _{\Delta_1}w_1\Lambda(d\vvw).
\]
\end{enumerate}
\end{Thm}
In the  regimes of non-critical speed, the limit Gaussian random fields are fractional Brownian sheets that have a direction with degenerate dependence, in the sense that the Hurst index in that direction is either $1/2$ (independent increments) or $1$ (complete dependence).
\begin{Rem}
For the boundary case between the two regimes of non-critical speed in Theorem \ref{thm:non_critical}, namely $n_1^{\alpha_1}\gg n_2^{\alpha_2}$ and $\alpha_1 = 1$, we expect
the following functional central limit theorem to hold
\equh\label{eq:alpha=1}
\frac1{\sqrt{n_1\log n_1}n_2\sqrt{m(\vvn)}} \ccbb{\what S_\vvn(\vvt)}_{\vvt\in[0,1]^2} \weakto \sigma \ccbb{\B_\vvt^\vvH}_{\vvt\in[0,1]^2},
\eque
with $\vvH = (1/2,1)$ and $\sigma^2 = 4\pi\int_{\Delta_1}w_1\Lambda(d\vvw)$. Note that when compared to the two regimes therein, while there is the continuous transition in terms of the Hurst indices $\vvH$, the normalization is inconsistent with the one in \eqref{eq:sub_regimes}, because of the extra logarithmic term. The analysis of this case is the most involved. However, in view of the limit, this is also the least interesting case as the limit random field has degenerate dependence in both directions. Therefore, we only prove
the convergence of covariance function for \eqref{eq:alpha=1}
in the last section of the Supplementary Material.
\end{Rem}

All the random fields in the limit are operator-scaling. For fractional Brownian sheet, it is well known that
\[
\ccbb{\B^\vvH_{\vv\lambda\cdot \vvt}}_{\vvt\ge \vv0}\eqd \lambda_1^{H_1}\lambda_2^{H_2}\ccbb{\B^\vvH_\vvt}_{\vvt\ge\vv 0},
\]
which is actually stronger than the operator-scaling property in \eqref{eq:OS}.
The limit random field $\G^{\vv\alpha,\Lambda}$ in Theorem \ref{thm:2} is also operator-scaling.

\begin{Prop}
For $\sccbb{\G_\vvt^{\vv\alpha,\Lambda}}_{\vvt\ge \vv0}$  in Theorem \ref{thm:2}, we have
$$
\ccbb{\G_{\lambda^{1/\alpha_1}t_1,\lambda^{1/\alpha_2}t_2}^{\vv\alpha,\Lambda}}_{\vvt\ge \vv0} \eqd \lambda^{1/\alpha_1+1/\alpha_2-1/2}\ccbb{\G_\vvt^{\vv\alpha,\Lambda}}_{\vvt\ge \vv0}, \mfa \lambda>0.
$$
\end{Prop}

\begin{proof}
Since $\ccbb{\G_\vvt^{\vv\alpha,\Lambda}}_{\vvt\ge 0}$ is a Gaussian random field, it suffices to show
$$
\cov(\G_{\lambda^{1/\alpha_1}s_1,\lambda^{1/\alpha_2}s_2}^{\vv\alpha,\Lambda}, \G_{\lambda^{1/\alpha_1}t_1,\lambda^{1/\alpha_2}t_2}^{\vv\alpha,\Lambda})=\lambda^{2/\alpha_1+2/\alpha_2-1}\cov(\G_\vvs^{\vv\alpha,\Lambda}, \G_\vvt^{\vv\alpha,\Lambda}).
$$
Define $\vv\theta'=(\theta_1',\theta_2'):=(\lambda^{1/\alpha_1}\theta_1,\lambda^{1/\alpha_2}\theta_2)$, then
$$
\prodd k12\frac{(e^{i\lambda^{1/\alpha_k}s_k\theta_k}-1)\wb{(e^{i\lambda^{1/\alpha_k}t_k\theta_k}-1)}}{|\theta_k|^2}=\lambda^{2/\alpha_1+2/\alpha_2}\prodd k12\frac{(e^{is_k\theta_k'}-1)\wb{(e^{it_k\theta_k'}-1)}}{|\theta_k'|^2}.
$$
For the function $\Psi_{\vv\alpha, \Lambda}$, we have
$$
\Psi_{\vv\alpha, \Lambda}(\vv\theta')=\lambda^{1-1/\alpha_1-1/\alpha_2}\Psi_{\vv\alpha,\Lambda}(\vv\theta)
$$
by change of variable $r\to\lambda r$. Applying these two identities to \eqref{eq:cov_G} completes the proof.
\end{proof}

The proofs of our results are based on estimates of asymptotics of second and fourth moments of the partial sums of each single random field $S_n$. However, except for the model with independent persistence, our estimates are by a different method from the one used in \citep{enriquez04simple} in one dimension. The method used there is essentially the time-domain approach for long-range dependence,  relying on the analysis of regular variation of the covariance function and the Karamata's theorem. This approach, however, cannot be easily adapted to two dimensions. Instead, we take the frequency-domain approach by working with Fourier transforms of the random fields.
\subsection{Discussions}

We conclude the introduction with a few remarks.
\begin{Rem}\label{rem:double_limit}
There are other types of limit theorems in the investigation of aggregated models. For ours, we can write
\equh\label{eq:single_limit}
\frac1{a(\vvn)\sqrt{m(\vvn)}}\what S_\vvn(\vvt) = \frac1{a(\vvn)}\sum_{\vvj\in[\vv1,\vvn\cdot\vvt]}\frac1{\sqrt{m(\vvn)}}\summ i1{m(\vvn)} X_\vvj^i.
\eque
Especially in econometrics literature, often the aggregated model is referred to the limit random field $\{\mathfrak X_\vvj\}_{\vvj\in\N^2}$ in the weak convergence
$m^{-1/2}\summ i1m X_\vvj\topp i \weakto \mathfrak X_\vvj, \vvj\in\N^2$,
and the investigation of the long-range dependence of the aggregation concerns the behavior of the covariance  function of $\mathfrak X$, or equivalently its spectral density near origin. One may then scale these aggregated random fields to obtain operator-scaling random fields indexed by $\vvt\in[0,1]^2$ via
\equh\label{eq:double_limit}
\frac1{a(\vvn)}\sum_{\vvj\in[\vv1,\vvn\cdot\vvt]} \mathfrak X_\vvj,
\eque
by appropriate choice of $a(\vvn)$. The limit theorems in the form of \eqref{eq:double_limit}  is referred to as {\em taking a double limit}, as one lets the number of copies in aggregation tend to infinity first (as $m\to\infty$), and then the size of the lattice tend to infinity (as $\vvn\to\infty$). The limit theorems in the form of \eqref{eq:single_limit} is referred to as {\em taking a single limit}.

\citet{enriquez04simple} established actually limit theorems by taking both single limit and double limit for the one-dimensional model. We only worked out the single limit here, which is more demanding to establish.
If we take the double limit for our aggregated model, we expect the limit random fields to remain the same in all cases in aforementioned theorems, as shown in one dimension in \citep{enriquez04simple}.
We are not aware of any other limit theorems for aggregated random fields for single limits. \end{Rem}
\begin{Rem}
Our aggregated random-field model can be viewed as with an infinite-dimensional parameter $\Lambda$ on $\Delta_1$ and $\vv\alpha\in(0,2)^2$, and hence it leads to a large flexible family of operator-scaling Gaussian random fields.
There are several recent limit theorems on operator-scaling Gaussian random fields. However, besides the fractional Brownian sheets, it is not easy to compare the limits from different models. This suggests that the counterparts of fractional Brownian motions in high dimensions are far from being unique, which is a challenge for investigation of long-range dependence in high dimensions.

For example, \citet{bierme17invariance} established limit theorems for another flexible family of operator-scaling Gaussian random fields, in the investigation of a different random-field model. The Gaussian random fields in the limit have covariance function
\[
\sigma^2\int_{\R^2}\pp{\prodd k12\frac{(e^{is_k\theta_k}-1)\wb{(e^{it_k\theta_k}-1)}}{|\theta_k|^2}}\frac1{(\log\psi(\vv\theta))^2}d\vv\theta,
\]
where $\psi$ is the logarithm of the characteristic function of certain multivariate stable distribution.
\citet{puplinskaite16aggregation} proposed another aggregated random-field model (in the sense of taking a double limit as in Remark \ref{rem:double_limit}), which may lead to both Gaussian and non-Gaussian stable limits. However,
when restricted to a fixed domain of attraction, their model is essentially determined by one parameter (see \citep[Eq.~(1.8)]{puplinskaite16aggregation}, where $\beta$ plays the similar role as $q$ in Enriquez's original model), and hence is less flexible than ours and the one in \citep{bierme17invariance}.

It is not immediately clear to us whether it is possible to relate limit Gaussian random fields in \citep{bierme17invariance,puplinskaite16aggregation} to ours, and we leave this question to further investigation.
\end{Rem}

\begin{Rem}Our statements are actually more general than those in the aforementioned papers, where the rates of the rectangular regions are essentially assumed in the form of $n_2 = n_1^\gamma$ for different choices of $\gamma$. We expect that assumptions therein can be generalized to the slightly more relaxed type here.
\end{Rem}
\begin{Rem}
Here we observe a {\em scaling-transition phenomenon}, that is, when the underlying rectangles of the partial-sum random fields grow at different speeds, different random fields may arise in the limit.
Such a phenomenon has been known in a few limit theorems for random fields in the literature recently \citep{puplinskaite16aggregation,puplinskaite15scaling,bierme17invariance}, while our result here is the first, %YZ1206
to the best of our knowledge,
 to investigate the boundary case between regimes of non-critical speed.
The scaling-transition phenomenon  is essentially due to the fact that the covariance function of the limit Gaussian random field, say $C(\vvs,\vvt)$, does not factorize into product form $C_1(s_1,t_1)C_2(s_2,t_2)$ in general, with the only exception when the random field is a fractional Brownian sheet.
\end{Rem}

In the rest of the paper, we prove Theorems \ref{thm:1}, \ref{thm:2} and \ref{thm:non_critical} in Sections \ref{sec:1}, \ref{sec:2} and \ref{sec:non_critical}, respectively. Some auxiliary proofs are left to the Supplementary Material.

\section{Proof of Theorem \ref{thm:1}}\label{sec:1}

The proof of Theorem~\ref{thm:1} is based on three estimates, which all are based on a single random field $S_\vvn$, not the aggregated one $\what S_\vvn$.
\begin{Lem}\label{lem:1}
Under the assumption of Theorem~\ref{thm:1},
\equh\label{eq:cov}
\cov  (S_\vvn(\vv s),S_\vvn(\vv t)) \sim \sigma_1^2
 n_1^{2H_1}n_2^{2H_2}\cov\pp{\B^{\vvH}_\vvs,\B^{\vvH}_\vvt}
\eque
as $\vvn\to\infty$,
and there exists a constant $C$ such that
\equh\label{eq:var_t}
\esp S_\vvn(\vvt)^2 \le Cn_1^{2H_1}n_2^{2H_2}t_1^{2H_1}t_2^{2H_2} \mfa \vvn\in\N^2, \vvt\in[0,1]^2,
\eque
and
\equh\label{eq:4th}
\esp S_\vvn(\vvt)^4 \le Cn_1^{2H_1+2}n_2^{2H_2+2} \mfa \vvn\in\N^2,\vvt\in[0,1]^2.
\eque
\end{Lem}
\begin{proof}
Observe that
\[
S_\vvn(\vvt) = \sum_{\vvj\in[\vv1,\vvn\cdot\vvt]}X_\vvj = \sum_{j_1=1}^{\floor{n_1t_1}}\varepsilon\topp1_{j_1}\sum_{j_2=1}^{\floor{n_2t_2}}\varepsilon\topp2_{j_2} = S_{\floor{n_1t_1}}\topp1S_{\floor{n_2t_2}}\topp2,
\]
where $S\topp k_n = \summ j1n \varepsilon_j\topp k$, $k=1,2$ are independent. Then,
\[
\cov(S_\vvn(\vvs),S_\vvn(\vvt)) = \prodd k12 \cov\pp{S_{n_k}\topp k(s_k),S_{n_k}\topp k(t_k)},
\]
and $\esp S_\vvn^r(\vvt) = \prodd k12\esp S_{n_k}\topp k(t_k)^r$. The corresponding estimates on $S\topp k, k=1,2$ have been obtained in~\citep{enriquez04simple}. More precisely, in the proof of Corollary 1 in \citep{enriquez04simple}, it was shown that
$$
\cov\pp{S_{n_k}\topp k(s_k),S_{n_k}\topp k(t_k)}\sim \frac{\Gamma(3-2H_k)}{H_k(2H_k-1)}n_k^{2H_k}\cov\pp{\B^{H_k}_{s_k}, \B^{H_k}_{t_k}}
$$
and
$$
\esp S_{n_k}\topp k(t_k)^r=O((n_kt_k)^{r+(2H_k-2)}), r\in 2\N.
$$
Taking the products for $k=1,2$ finishes the proof.
\end{proof}

\begin{proof}[Proof of Theorem~\ref{thm:1}]
We first prove the convergence of finite-dimensional distributions.  It suffices to show, for all $d\in\N$, $a_1,\dots,a_d\in\R, \vv t_1,\dots,\vvt_d\in\R_+^2$,
\equh\label{eq:fdd}
\frac1{n_1^{H_1}n_2^{H_2}\sqrt {m(\vvn)}}\summ w1d a_w \what S_\vvn (\vvt_w) \weakto \sigma_1\summ w1da_w\B^{\vvH}_{\vvt_w}.
\eque
Observe that the right-hand side is a centered Gaussian random variable. At the same time, the left-hand side can be expressed as
\[
\frac1{\sqrt{m(\vvn)}}\summ i1{m(\vvn)}\frac1{n_1^{H_1}n_2^{H_2}}\summ w1d a_wS_\vvn^i(\vvt_w).
\]
By Lindeberg--Feller central limit theorem for triangular arrays of i.i.d.~random variables, to show~\eqref{eq:fdd} it suffices to show, for
\equh\label{eq:Yn}
Y_\vvn:= \frac1{n_1^{H_1}n_2^{H_2}} \summ w1d a_wS_\vvn(\vvt_w),
\eque
\equh\label{eq:LF1}
\lim_{\vvn\to\infty}\var\pp{Y_\vvn} = \sigma_1^2\var\pp{\summ w1da_w\B^{\vvH}_{\vvt_w}},
\eque
and
\equh\label{eq:LF2}
\lim_{\vvn\to\infty}\esp\pp{Y_\vvn^2\inddd{Y_\vvn^2>m(\vvn)\eta}} = 0 \mfa \eta>0.
\eque

For \eqref{eq:LF1}, Write
\[
\var(Y_\vvn) = \frac{1}{n_1^{2H_1}n_2^{2H_2}}\summ w1d\summ{w'}1d a_wa_{w'}\cov(S_\vvn(\vvt_w),S_\vvn(\vvt_{w'})),
\]
and similarly for $\var(\summ w1d a_w\B^{\vvH}_{\vvt_w})$. Then,~\eqref{eq:LF1} follows from~\eqref{eq:cov}.

Next, we prove~\eqref{eq:LF2}. Observe that by Markov inequality and~\eqref{eq:4th},
\begin{align*}
\esp \pp{Y_\vvn^2\inddd{Y_\vvn^2>m(\vvn)\eta}} & \leq \frac1{m(\vvn)\eta}\esp Y_\vvn^4
 \le
 \frac1{m(\vvn)\eta}\pp{\frac1{n_1^{H_1}n_2^{H_2}}\summ w1d |a_w| \pp{\esp S_\vvn(\vvt_w)^4}^{1/4}}^4 \nonumber\\
 & \leq \frac C{m(\vvn)\eta n_1^{4H_1}n_2^{4H_2}}\summ w1d|a_w| n_1^{2H_1+2}n_2^{2H_2+2}= \frac C\eta\frac {n_1^{2-2H_1}n_2^{2-2H_2}}{m(\vvn)}.\nonumber
\end{align*}
Therefore,~\eqref{eq:LF2} is satisfied, under the assumption~\eqref{eq:m(n)}.

Next, we prove the tightness. By \citep[Theorem 3 and remark afterwards]{bickel71convergence}, it suffices to show that there exist $p\in\N, \gamma_1,\gamma_2>1, C>0$ such that
\equh\label{eq:Bickel}
\esp\abs{\frac{\what S_\vvn(\vvt)}{n_1^{H_1}n_2^{H_2}\sqrt{m(\vvn)}}}^{2p} \le C t_1^{\gamma_1}t_2^{\gamma_2}, \mfa \vvn\in\N^2, \vvt\in\R_+^2.
\eque
For this purpose, observe that
\[
\esp\what S_\vvn(\vvt)^2 = m(\vvn)\esp S_\vvn(\vvt)^2 \le Cm(\vvn)(n_1t_1)^{2H_1}(n_2t_2)^{2H_2}
\]
because of~\eqref{eq:var_t}.
The tightness thus follows.
\end{proof}
As the above proof shows, the functional central limit theorem is essentially based on the three estimates in Lemma~\ref{lem:1}. The functional central limit theorems for other models will be very similarly based on corresponding estimates moments. For the model with independent persistence, these estimates are  almost immediate from the one-dimensional ones in \citep{enriquez04simple}. However, for the model with dependent persistence, the one-dimensional estimates can no longer be used, and we have to take a completely different approach.
\section{Proof of Theorem \ref{thm:2}}\label{sec:2}
Throughout, we restrict ourselves to the aggregated model with dependent persistence, with
\[
\alpha_1,\alpha_2\in(0,2),
\]
and that
\equh\label{eq:n*}
n^*:=n_1^{\alpha_1}\sim n_2^{\alpha_2} \mmas \vvn\to\infty,
\eque
which we shall assume in this section without further mentioning. Some of our estimates are universal and do not depend on this assumption, and in this case we will say explicitly ``for all $\vvn\in\N^2$''.
We write also
\[
p(\vv\alpha) = \frac1{\alpha_1}+\frac1{\alpha_2}
\]
in the sequel.

We start with the computation of the asymptotic covariance.
\begin{Prop}\label{prop:cov_I_AD}
We have
\[
\lim_{\vvn\to\infty}\frac{\cov\pp{S_\vvn(\vvs),S_\vvn(\vvt)}}{|\vvn|^2/n^*} = \cov(\G_\vvs^{\vv\alpha,\Lambda},\G^{\vv\alpha,\Lambda}_\vvt),
\] and there exists a constant $C$ such that
\equh\label{eq:bickel}
\esp S_\vvn(\vvt)^2 \le C\frac{|\vvn|^2}{n^*}(t_1t_2)^{2-1/p(\vv\alpha)}
\eque
for all $\vvt\in[0,1]^2$ such that $\floor{\vvn\cdot \vv t} = \vvn\cdot\vv t$.

\end{Prop}
The two estimates above are obtained by computing the Fourier transforms of the covariance.
 For background on multidimensional Fourier transforms, see \cite{pinsky02introduction}.

Let $r$ denote the covariance function of the stationary random field
\[
r(\vv \ell) = \cov(X_{\vv1},X_{\vv1+\vvl}),
\]
and $\what r(\vv\theta) := \sum_{\vv\ell\in\Z^2}r(\vvl)\exp(i\ip{\vv\ell\cdot\vv\theta})$ its Fourier transform. Introduce
 the Fourier transform of the sequence
 $\{a_j\}_{j\in\N} = \{\indd{1\leq j\leq n}\}_{j\in\N}$
\[
D_n(\theta):= \summ j1{n}e^{ij\theta},
\]
and
set
\[
 D_{\vvn,\vvs,\vvt}(\vv\theta):= \prodd k12D_{\sfloor{n_ks_k}}\pp{\theta_k}\wb{D_{\sfloor{n_kt_k}}\pp{\theta_k}}.
\]
\begin{Lem}\label{lem:covariance}We have
\equh
\cov(S_\vvn(\vvs),S_\vvn(\vvt))
 = \frac1{(2\pi)^2}\int_{(-\pi,\pi)^2}D_{\vvn,\vvs,\vvt}(\vv\theta)\,\what r\pp{\vv\theta}d\vv\theta.\label{eq:cov1}
\eque
\end{Lem}
\begin{proof}
To see this, we first write
\equh
\cov(S_\vvn(\vvs),S_\vvn(\vvt)) = \sum_{\vvi\in[\vv1,\vvn\cdot \vvs]}\sum_{\vvj\in[\vv1,\vvn\cdot\vvt]}\cov(X_{\vvi},X_{\vvj})  = \sum_{\vvl\in\Z^2}r(\vvl)\sum_{\vvj\in\Z^2}\inddd{\vvj\in[\vv1,\vvn\cdot\vvs],\vvj+\vvl\in[\vv1,\vvn\cdot\vvt]}.\label{eq:cov_Fourier}
\eque
Introduce $a_\vvj = \indd{\vvj\in[\vv1,\vvn\cdot\vvs]}, b_\vvj = \indd{\vvj\in[\vv1,\vvn\cdot\vvt]}, \vvj\in\Z^2$, and let $\what a(\vv\theta)$ and $\what b(\vv\theta)$ denote their Fourier transforms, respectively. Then, for each $\vvl\in\Z^2$, we have
\[
\sum_{\vvj\in\Z^2}\indd{\vvj\in[\vv1,\vvn\cdot\vvs],\vvj+\vvl\in[\vv1,\vvn\cdot\vvt]} = \sum_{\vvj\in\Z^2}a_\vvj b_{\vvj+\vvl},
\]
which is the $\vvl$-th coefficient of $\wb{\what a(\vv\theta)}\what b(\vv\theta)$. We have that
\[
\what a(\vv\theta) = \prodd k12D_{\sfloor{n_ks_k}}(\theta_k) \qmand \what b(\vv\theta) = \prodd k12D_{\sfloor{n_kt_k}}(\theta_k).
\]
So by Parseval's theorem,~\eqref{eq:cov_Fourier} becomes
\[
\cov(S_\vvn(\vvs),S_\vvn(\vvt)) =  \frac1{(2\pi)^2}\int_{(-\pi,\pi)^2}\wb{\wb{\what a(\vv\theta)}\what b(\vv\theta)} \what r(\vv\theta)d\vv\theta,
\]
which yields~\eqref{eq:cov1}.
\end{proof}

The next step is to apply a change of variables
\[
\vv\theta\to\frac{\vv\theta}{\vvn}:=\pp{\frac{\theta_1}{n_1},\frac{\theta_2}{n_2}},
\] and hence to write
\equh
\cov  (S_\vvn(\vv s),S_\vvn(\vv t)) = \frac{1}{|\vvn|(2\pi)^2}\int_{\vvn\cdot(-\pi,\pi)^2}{D_{\vvn,\vvs,\vvt}(\vv\theta/\vvn)}\what r\pp{\vv\theta/\vvn}d\vv\theta.\label{eq:cov_DCT}
\eque
The two functions of the integrand can then be treated separately. The following results on $D_{\vvn,\vvs,\vvt}$ are well known and provided here only for the sake of completeness. In the sequel, we write
\[
\R_o^2 = (\R\setminus\{0\})^2.
\]
\begin{Lem}\label{lem:D}
In the notations above,
\[
\lim_{\vvn\to\infty}\frac{D_{\vvn,\vvs,\vvt}(\vv\theta/\vvn)}{|\vvn|^2} = \prodd k12 \frac{(e^{is_k\theta_k}-1)\wb{(e^{it_k\theta_k}-1)}}{|\theta_k|^2} \mfa \vv\theta\in\R_o^2
\]
and
\equh\label{eq:D_upper}
\abs{\frac{D_{\vvn,\vvs,\vvt}(\vv\theta/\vvn)}{|\vvn|^2}} \le \pi^2\prodd k12\min\ccbb{s_kt_k,\frac1{|\theta_k|^2}}, \mfa \vvn\in\N^2, |\theta_k|\le n_k\pi.
\eque
\end{Lem}
\begin{proof}
It is easy to show that
\[
\limn \frac1n{D_{\floor{nt}}}\pp{\frac \theta n} = \frac{e^{it\theta}-1}{i\theta},
\]
and, because of $|\sin(x)|\ge 2|x|/\pi$ for $|x|\le \pi/2$, and $|\sin x|\le \min(|x|,1)$,
\equh\label{eq:D}
\abs{\frac1n D_{\floor{nt}}\pp{\frac\theta n}} = \abs{\frac{\sin(\floor{nt}\theta/(2n))}{n\sin(\theta/(2n))}}\le \pi \min\ccbb{t, \frac1{|\theta|}}, n\in\N, |\theta|\le n\pi.
\eque
The desired results now follow.
\end{proof}
Most of the effort will be devoted to the analysis of $r$ and $\what r$.
\begin{Lem}For  $\vv\theta\in(-\pi,\pi)^2$ such that $\theta_1\neq 0, \theta_2\neq 0$,
 \equh\label{eq:r_Fourier}
 \what r(\vv\theta)
=\int G^*(\vvu,\vv\theta)\mu^*(d\vv u) \mwith G^*(\vvu,\vv\theta):=  \prodd k12\frac{u_k(2-u_k)}{u_k^2+2(1-u_k)(1-\cos\theta_k)}.
\eque
\label{lem:r_hat}Moreover,
\[
\what r(\vv\theta/\vvn)\sim \frac{|\vvn|}{n^*}\Psi_{\vv\alpha,\Lambda}(\vv\theta)
\]
as $\vvn\to\infty$, and there exists a constant $C$ such that
\equh\label{eq:bound_rhat}
\what r(\vv\theta/\vv n) \le C\pp{\frac{|\vv n|}{n^*}|\theta_1|^{1/p(\vv\alpha)-1} + \frac{n_1}{n^*}|\theta_1|^{\alpha_1-1} +\frac{n_2}{n^*}|\theta_2|^{\alpha_2-1} + 1}.
\eque
for all $\vv\theta\in\R_o^2$.
\end{Lem}
\begin{proof}
 We have
\begin{align*}
r(\vvl) & = \esp(X_{\vv1}X_{\vv1+\vvl}) = \esp\pp{\varepsilon_1\topp1\varepsilon_{1+\ell_1}\topp1\varepsilon_1\topp2\varepsilon_{1+\ell_2}\topp2}  = \esp\bb{\esp\pp{\varepsilon\topp1_1\varepsilon\topp1_{1+\ell_1}\mmid q_1}\esp\pp{\varepsilon\topp2_1\varepsilon\topp2_{1+\ell_2}\mmid q_2}}\\
&= \esp\bb{(2q_1-1)^{|\ell_1|}(2q_2-1)^{|\ell_2|}} = \int(1-u_1)^{|\ell_1|}(1-u_2)^{|\ell_2|}\mu^*(d\vv u).
\end{align*}
Consider
\[
\what r(\vv\theta)  = \sum_{\vvl\in\Z^2}r(\vvl)e^{i\ip{\vvl,\vv\theta}} = \sum_{\vvl\in\Z^2} \int e^{i\ip{\vvl,\vv\theta}}(1-u_1)^{|\ell_1|}(1-u_2)^{|\ell_2|}\mu^*(d\vv u).
\]
Recall that
\[
\sum_{\ell\in\Z}\rho^{|\ell|}e^{i\ell\theta} = \frac{1-\rho^2}{1-2\rho\cos \theta+\rho^2}, \mfa \rho\in(-1,1).
\]
So~\eqref{eq:r_Fourier} follows.

Now we investigate the asymptotics of $\what r$. Recall that we let $\mu^*$ denote the measure on $(0,1]^2$ induced by $\vv U$.
It turns out to be  convenient to  work with polar coordinates. For this purpose, introduce
\[
 T_{\vv\alpha}(\vv x):=\pp{\frac1{x_1^{\alpha_1}},\frac1{x_2^{\alpha_2}}}.
\]
So $\mu^*\circ  T_{\vv\alpha}\inv$ is the measure on $[1,\infty)^2$ induced by $R\vv W$, and for any measurable function $f:\R_+^2\to\R$,
\equh\label{eq:RW}
\int_{T_{\vv\alpha}((0,1]^2)} f(\vv x)\mu^*\circ T_{\vv\alpha}\inv(d\vv x) = \int_1^\infty\int_{\Delta_1}\inddd{r\vvw\in T_{\vv\alpha}((0,1]^2)}f(r\vvw)\Lambda(d\vvw)r^{-2}dr,
\eque
provided the integrability of either side can be justified.
We treat $\vv U = (1,1)$ and $\vv U\in(0,1)^2$ separately, and write
\begin{align*}%\label{eq:r1}
\what r(\vv\theta/\vvn) & = \int_{T_{\vv\alpha}((0,1]^2)}G^*\pp{T_{\vv\alpha}\inv(\vvu),\vv\theta/\vvn}\mu^*\circ T_{\vv\alpha}\inv(d\vvu) + \proba(\vv U = (1,1))
\\&=: \what r_1(\vv\theta/\vvn)+\proba(\vv U = (1,1)).
\end{align*}
We shall see eventually that $\what r_1(\vv\theta/\vvn)$ is of order $|\vvn|/n_1^{\alpha_1}$, so $\what r(\vv\theta/\vvn)\sim\what r_1(\vv\theta/\vvn)$. We focus on $\what r_1(\vv\theta/\vvn)$ from now on. Recall that $n^* = n_1^{\alpha_1}$. Note that
\begin{align}
\what r_1(\vv\theta/\vvn) &= \int_1^\infty\int_{\Delta_1}\inddd{r\vvw\in T_{\vv\alpha}((0,1]^2)}G^*\pp{T_{\vv\alpha}\inv(r\vvw),\vv\theta/\vvn}\Lambda(d\vvw)r^{-2}dr
\nonumber\\
& = \frac1{n^*}\int_0^\infty \int_{\Delta_1}\inddd{n^*r\vvw\in  T_{\vv\alpha}((0,1]^2)}G^*\pp{T_{\vv\alpha}\inv(n^*r\vvw),\vv\theta/\vvn}\Lambda(d\vvw)r^{-2}dr.\label{eq:Phi_n}
\end{align}
In the last line above, we first applied a change of variables, and then replaced  $\int_{1/n^*}^\infty$ by  $\int_0^\infty$, as the constraint $n^*r\vvw\in T_{\vv\alpha}((0,1)^2)$ implies that $r\ge (n^*w_k)\inv\ge (n^*)\inv$. Introduce
\[
h_{\vvn}(r,\vvw,\vv\theta) :=G^*\pp{T_{\vv\alpha}\inv (n^*r\vvw),\vv\theta/\vvn} \inddd{n^*r\vvw\in T_{\vv\alpha}((0,1)^2)}.
\]
In view of integral expressions~\eqref{eq:Psi} and~\eqref{eq:Phi_n}, to show the first part of the lemma we need to prove, for
\[
h(r,\vvw,\vv\theta):= \prodd k12 \frac{2(rw_k)^{-1/\alpha_k}}{(rw_k)^{-2/\alpha_k}+\theta_k^2},
\]
that
\begin{multline}\label{eq:hnK}
\frac{n^*}{|\vvn|}\what r_1(\vv\theta/\vvn) \equiv  \frac1{|\vvn|}\int_0^\infty \int_{\Delta_1} h_{\vvn}(r,\vvw,\vv\theta)\Lambda(d\vvw)r^{-2}dr\\
\to\Psi_{\vv\alpha,\Lambda}(\vv\theta) \equiv \int_0^\infty \int_{\Delta_1}h(r,\vvw,\vv\theta)\Lambda(d\vvw)r^{-2}dr \in(0,\infty) \mmas \vvn\to\infty.
\end{multline}
For this purpose, we show, for any $\delta\in(0,1)$,
\begin{multline}\label{eq:DCT_delta}
\lim_{\vvn\to\infty} \frac1{|\vvn|}\int_0^\infty\int_{\Delta_1}h_{\vvn}(r,\vvw,\vv\theta)\indd{n^*r\vvw\in T_{\vv\alpha}((0,\delta]^2)}\Lambda(d\vvw)r^{-2}dr \\
= \int_0^\infty \int_{\Delta_1}h(r,\vvw,\vv\theta)\Lambda(d\vvw)r^{-2}dr,
\end{multline}
and for some constant $C$ independent of $\vv\theta$,
\begin{multline}\label{eq:lower_order}
\int_0^\infty\int_{\Delta_1}h_\vvn(r,\vvw,\vv\theta)\indd{n^*r\vvw\in T_{\vv\alpha}((0,1]^2\setminus(0,\delta]^2)}\Lambda(d\vvw)r^{-2}dr\\ \le C(n_1|\theta_1|^{\alpha_1-1}+n_2|\theta_2|^{\alpha_2-1}+n^*).
\end{multline}
Thus the integral in \eqref{eq:lower_order} does not contribute in the limit, since $|\vvn|\sim (n^*)^{p(\vv\alpha)}$ as $\vvn\to\infty$ and $p(\vv\alpha)>1$.

We first show~\eqref{eq:DCT_delta}.
By the definition of $G^*$, we have
\begin{align*}
 h_{\vvn,\delta}(r,\vvw,\vv\theta)& := h_\vvn(r,\vvw,\vv\theta)\inddd{n^*r\vvw\in T_{\vv\alpha}((0,\delta]^2)} =G^*\pp{T_{\vv\alpha}\inv (n^*r\vvw),\vv\theta/\vvn} \inddd{n^*r\vvw\in T_{\vv\alpha}((0,\delta]^2)}\\
&  = \prodd k12 g\pp{(n^*rw_k)^{-1/\alpha_k},\theta_k/n_k}\inddd{n^*r\vvw\in T_{\vv\alpha}((0,\delta]^2)}
\end{align*}
with
\[
g(u,\theta) := \frac{u(2-u)}{u^2+2(1-u)(1-\cos\theta)}.
\]
It is clear that
for every $r>0, \vvw\in\Delta_1$, $\vv\theta\in\R_o^2$,
\[
\lim_{\vvn\to\infty} \frac1{|\vvn|}h_{\vvn}(r,\vvw,\vv\theta)  = h(r,\vvw,\vv\theta).
\]
So to prove~\eqref{eq:DCT_delta}, by the dominated convergence theorem it remains to find an integrable upper bound of $h_{\vvn,\delta}/|\vvn|$. For this purpose,
observe that by the trivial bound $g(u,\theta)\le 2u\inv$,
\equh\label{eq:g_bound1}
g\pp{(n^*rw)^{-1/\alpha_k},\theta_k/n_k}\le 2(n^*rw_k)^{1/\alpha_k},
\eque
 and that, recalling the fact $2(1-\cos\theta)= 4\sin^2(\theta/2)\ge 4\theta^2/\pi^2$ for $\theta\in(-\pi,\pi)$,
  \equh\label{eq:<delta}
 g\pp{(n^*rw_k)^{-1/\alpha_k},\theta_k/n_k}\inddd{(n^*rw_k)^{-1/\alpha_k}\in(0,\delta)}\le \frac{(n^*rw_k)^{-1/\alpha_k}}{2(1-\delta)\theta_k^2n_k^{-2}/\pi^2} = \frac {C_\delta n_k}{(rw_k)^{1/\alpha_k}\theta_k^2}
 \eque
 for some constant $C_\delta$ depending only on $\delta$. Here we used the fact that there exists universal constants
 $c_1,c_2$ such that $c_1n_2^{\alpha_2}\le n^*\le c_2n_2^{\alpha_2}$ for the sequence $\vvn$ of our interest.
 Therefore,
\[
\frac1{|\vvn|}h_{\vvn,\delta}(r,\vvw,\vv\theta) \le C_\delta\prodd k12 \min\ccbb{(rw_k)^{1/\alpha_k},\frac1{(rw_k)^{1/\alpha_k}\theta_k^2}} =:C_\delta\wb h(r,\vvw,\vv\theta).
\]
We now show that $\iint \wb h(r,\vvw,\vv\theta)\Lambda(d\vvw)r^{-2}dr<\infty$. Introduce
$$
d(\vv \theta, \vvw, \vv \alpha):=|\theta_1\theta_2|
^{-1/p(\vv\alpha)}
w_1^{-\alpha_2/(\alpha_1+\alpha_2)}w_2^{-\alpha_1/(\alpha_1+\alpha_2)}.
$$
Then
\begin{multline*}
\iint \wb h(r,\vvw,\vv\theta)\Lambda(d\vvw)r^{-2}dr\leq \int_{\Delta_1}\int_0^{d(\vv\theta, \vvw, \vv \alpha)}r^{p(\vv\alpha)-2}drw_1^{1/\alpha_1}w_2^{1/\alpha_2}\Lambda(d\vvw)\\
+\int_{\Delta_1}\int_{d(\vv\theta, \vvw, \vv \alpha)}^\infty r^{-p(\vv\alpha)-2}drw_1^{-1/\alpha_1}w_2^{-1/\alpha_2}\frac{1}{|\theta_1\theta_2|}\Lambda(d\vvw).
\end{multline*}
It can be easily verified that each double integral on the right-hand side above is bounded by
$$
C|\theta_1\theta_2|
^{1/p(\vv\alpha)-1}
\int_{\Delta_1}w_1^{\alpha_2/(\alpha_1+\alpha_2)}w_2^{\alpha_1/(\alpha_1+\alpha_2)}\Lambda(d\vvw)\leq C|\theta_1\theta_2|
^{1/p(\vv\alpha)-1}.
$$
Namely,  there exists a constant $C$ depending only on $\Lambda,\alpha_1,\alpha_2$, such that
\[
\int_{\R_+\times\Delta_1}\wb h(r,\vvw,\vv\theta)\Lambda(d\vvw)r^{-2}dr\le C{|\theta_1\theta_2|^{1/p(\vv\alpha)-1}}.
\]
So we have proved~\eqref{eq:DCT_delta} and
\equh\label{eq:hn_delta}
\int_0^\infty\int_{\Delta_1}h_{\vvn,\delta}(r,\vvw,\vv\theta)\Lambda(d\vvw)r^{-2}dr \le C|\vvn||\theta_1\theta_2|^{1/p(\vv\alpha)-1}.
\eque
Now we prove \eqref{eq:lower_order}. We shall divide the region $\{n^*r\vvw\in T_{\vv\alpha}((0,1]^2\setminus (0,\delta]^2)\}$ into three pieces and treat each corresponding integral respectively. First,
for $u\in(\delta,1]$, we have $g(u,\theta)\le 2/\delta^2$, and hence
\[
h_\vvn(r,\vvw,\vv\theta)\indd{n^*r\vvw\in T_{\vv\alpha}((\delta,1]^2)}\le C_\delta.
\]
 Thus,
\begin{align}
\int_0^\infty& \int_{\Delta_1}  h_\vvn(r,\vvw,\vv\theta)\indd{n^*r\vvw\in T_{\vv\alpha}((\delta,1]^2)}\Lambda(d\vvw)r^{-2}dr  \nonumber\\
& \le C_\delta \int_{\Delta_1}\int_{(n^*)\inv(w_1\inv\vee w_2\inv)}^{(n^*)\inv[(w_1\delta^{\alpha_1})\inv\wedge (w_2\delta^{\alpha_2})\inv]}  r^{-2}dr\Lambda(d\vvw) \le C_\delta n^* \int_{\Delta_1} w_1\wedge w_2\Lambda(d\vvw)\nonumber \\
& \le C_\delta n^*,\label{eq:hn_delta1}
\end{align}
Similarly,
\begin{align*}
h_{\vvn}(r,\vvw,\vv\theta)\inddd{n^*r\vv w\in T_{\vv\alpha}((0,\delta]\times(\delta,1])} & \le C_\delta g\pp{(n^*rw_1)^{-
1/\alpha_1},\frac{\theta_1}{n_1}}\inddd{n^*rw_1%YS1206 changed >
\geq\delta^{-\alpha_1}}\\
& \le C_\delta n_1\min\ccbb{(rw_1)^{1/\alpha_1},\frac1{(rw_1)^{1/\alpha_1}\theta_1^2}}.
\end{align*}
This time, taking $d_1(\theta,w,\alpha) := |\theta|^{-\alpha}w\inv$,
and writing
\[
\int_{\Delta_1}\int_0^\infty = \int_{\Delta_1}\pp{\int_0^{d_1(\theta_1,w_1,\alpha_1)} + \int_{d_1(\theta_1,w_1,\alpha_1)}^\infty},
\] we have
\equh\label{eq:hn_delta2}
\int_0^\infty \int_{\Delta_1}  h_\vvn(r,\vvw,\vv\theta)\inddd{n^*r\vvw\in T_{\vv\alpha}((0,\delta]\times (\delta,1])}\Lambda(d\vvw)r^{-2}dr  \le C_\delta n_1|\theta_1|^{\alpha_1-1}\int_{\Delta_1}w_1\Lambda(d\vvw).
\eque
By symmetry, a similar bound holds for the left-hand side above with the indicator function replaced by $\inddd{n^*r\vv w\in T_{\vv\alpha}((\delta,1]\times(0,\delta])}$. Now \eqref{eq:bound_rhat} follows from combining \eqref{eq:hnK}, \eqref{eq:hn_delta}, \eqref{eq:hn_delta1} and \eqref{eq:hn_delta2}.
\end{proof}

\begin{proof}[Proof of Proposition~\ref{prop:cov_I_AD}]
Recall~\eqref{eq:cov_DCT}.
By Lemmas~\ref{lem:D} and~\ref{lem:r_hat}, the dominated convergence theorem yields the first part of the proposition, and that the integral in~\eqref{eq:cov_G} is finite (recalling the assumption that $\alpha_1,\alpha_2\in(0,2)$). The details are omitted. For the second part, it also follows from Lemma~\ref{lem:r_hat},
\eqref{eq:cov_DCT}
 and~\eqref{eq:D_upper} that
\begin{multline}
\esp S_\vvn(\vvt)^2
\le  C|\vvn|\int_{\vvn\cdot(-\pi,\pi)^2}\prodd k12\min\ccbb{t_k^2,\frac1{|\theta_k|^2}} \\
\times
\pp{\frac{|\vv n|}{n^*}|\theta_1\theta_2|^{1/p(\vv\alpha)-1} + \frac{n_1}{n^*}|\theta_1|^{\alpha_1-1} +\frac{n_2}{n^*}|\theta_2|^{\alpha_2-1} + 1}d\vv\theta.\label{eq:Snt2ndM}
\end{multline}
By change of variables, the above integral is bounded by
\[t_1t_2\int_{\R^2}\prodd k12\min\ccbb{1,\frac1{\theta_k^2}}\pp{\frac{|\vv n|}{n^*}\abs{\frac{\theta_1\theta_2}{t_1t_2}}^{1/p(\vv\alpha)-1} + \frac{n_1}{n^*}\abs{\frac{\theta_1}{t_1}}^{\alpha_1-1} +\frac{n_2}{n^*}\abs{\frac{\theta_2}{t_2}}^{\alpha_2-1} + 1}d\vv\theta.
\]
 Therefore, it follows that, for a constant $C$ independent of $n$ and $\vvt$,
\[
\esp S_\vvn(\vvt)^2 \le C\pp{\frac{|\vvn|^2}{n^*}(t_1t_2)^{2-1/p(\vv\alpha)} + \frac{n_1{|\vvn|}}{n^*}t_1^{2-\alpha_1}t_2 + \frac{n_2{|\vvn|}}{n^*}t_1t_2^{2-\alpha_2}+ |\vvn|t_1t_2}.\]
For $\vvn\in\N^2$ and $\vvt\in[0,1]^2$ such that $\vvn\cdot\vvt = \sfloor{\vvn\cdot\vvt}$, we have
\begin{align*}
\frac{n_1|\vv n|}{n^*}t_1^{2-\alpha_1}t_2 & = \frac{n_1^2}{n^*}t_1^{2-\alpha_1} \cdot n_2t_2\le C(n_1t_1)^{2-\alpha_1}(n_2t_2)^{2-1/p(\vv\alpha)} \\
& \le C(|\vvn|t_1t_2)^{2-1/p(\vv\alpha)} \le C\frac{|\vvn|^2}{n^*}(t_1t_2)^{2-1/p(\vv\alpha)}
\end{align*}
and
\[
|\vvn|t_1t_2 \le |\vvn|^{2-1/p(\vv\alpha)}(t_1t_2)^{2-1/p(\vv\alpha)} \le C\frac{|\vvn|^2}{n^*}(t_1t_2)^{2-1/p(\vv\alpha)},
\]
where we used the assumption \eqref{eq:n*}.
We have thus obtained the second part of the proposition.
\end{proof}

The second estimate that we need  is on the fourth moment.
\begin{Prop}\label{prop:typeI_4th}
There exists a constant $C$ such that
\[
\esp S_\vvn(\vvt)^{4}\le C \frac{|\vvn|^4}{n^*}(t_1t_2)^{4-1/p(\vv\alpha)}
\]
for all $\vvn\in\N^2, \vvt\in[0,1]^2$, such that $\floor{\vvn\cdot\vv t} = \vvn\cdot\vvt$.
\end{Prop}
\begin{proof}
Writing $\esp_q(\cdot) = \esp(\cdot\mid q)$, we have
\[
\esp S_\vvn(\vvt)^{4}=\esp\pp{\sum_{\vvi\in[\vv1,\vvn\cdot\vvt]}X_\vvi}^{4} = \esp\bb{\prodd k12\esp_{q_k}\pp{\summ{i_k}1{\sfloor{n_kt_k}}\varepsilon_{i_k}\topp k}^{4}}.
\]
Note that, for $\indn\varepsilon$ from the one-dimensional Enriquez model with random parameter $q$, $\esp_qS_n^{4} \le C\sum_{1\le i_1\le\cdots\le i_{4}\le n}\esp_q(\varepsilon_{i_1}\cdots\varepsilon_{i_{4}})$, and
\begin{multline*}
 \sum_{1\le i_1\le \cdots\le i_{4}\le n}  \esp_q\pp{\varepsilon_{i_1}\cdots\varepsilon_{i_{4}}}  = \sum_{\substack{j_1,j_2\ge 0,
 k_1,k_2\ge 0\\j_1 +j_2 +k_1+k_2\le n-1}}(2q-1)^{j_1+j_2}\\
  = \summ \ell 0{n-1} (2q-1)^\ell {\sum_{\substack{j_1,j_2\ge0\\ j_1+j_2 = \ell}}\sum_{\substack{k_1,k_2\ge 0\\k_1+k_2\le n-1-\ell}}1}  = \summ \ell 0{n-1}  (2q-1)^\ell \binom {\ell+1}{1}\binom{n+1-\ell}2.
\end{multline*}
So, for some constant $C$,
\[
\sum_{1\le i_1\le \cdots\le i_{4}\le n}  \esp_q\pp{\varepsilon_{i_1}\cdots\varepsilon_{i_{4}}}
\le C \summ\ell{-n}n(2q-1)^{|\ell|} |\ell|(n-|\ell|)^2.
\]
Thus,
\begin{align*}
\esp S_\vvn(\vvt)^{4} & \le C\int\prodd k12{\sum_{\ell_k=-\sfloor{n_kt_k}}^{\sfloor{n_kt_k}}\bb{|\ell_k|(\sfloor{n_kt_k}-|\ell_k|)^2(2q_k-1)^{|\ell_k|}}}\mu(d\vv q)\nonumber\\
& =C\int\sum_{\vv\ell\in[-\vvn\cdot\vv t,\vvn\cdot\vvt]}\prodd k12\bb{|\ell_k|(\sfloor{n_kt_k}-|\ell_k|)^2(2q_k-1)^{|\ell_k|}}\mu(d\vv q).
\end{align*}
Introduce
\[
J^{*}_\vvn(\vv\theta):=\sum_{\vv\ell\in[-\vvn,\vvn]}\prodd k12\bb{|\ell_k|(n_k-|\ell_k|)^2}e^{i\ip{\vvl,\vv\theta}}
= \prodd k12 J_{n_k}(\theta_k),\vvn\in\N^2,
\]
with
\[
J_n(\theta):= \sum_{\ell=-n}^n |\ell|(n-|\ell|)^2e^{i\ell\theta}.
\]
In summary,
\begin{align}
\esp S_\vvn(\vvt)^4 & \le  C\sum_{\vv\ell\in[-\vvn\cdot\vvt,\vvn\cdot\vvt]}\prodd k12\bb{|\ell_k|(\sfloor{n_kt_k}-|\ell_k|)^2}\int\prodd k12 (1-u_k)^{|\ell_k|}\mu^*(d\vvu)\nonumber\\
& = \frac C{(2\pi)^2}\int_{(-\pi,\pi)^2}J^{*}_{\floor{\vvn\cdot\vvt}}(\vv\theta) \what r(\vv\theta)d\vv\theta \nonumber\\%\label{eq:4th_0}\\
& =
 \frac{C}{|\vvn|(2\pi)^2}\int_{\vvn\cdot(-\pi,\pi)^2}J^{*}_{\floor{\vvn\cdot\vvt}}(\vv\theta/\vvn) \what r(\vv\theta/\vvn)d\vv\theta.\label{eq:4th_1}
\end{align}
Now we establish the following bound: for some constant $C$,
\equh\label{eq:Hnp_DCT}
\abs{J_{\sfloor{nt}}\pp{\frac\theta n}}\le Cn^{4}t^2\min\ccbb{t^2,\frac1{|\theta|^2}} \mfa n\in\N, t\in[0,1], \theta\in(-n\pi,n\pi).
\eque
Observe that this bound and Lemma~\ref{lem:r_hat} yield the desired result.

To show~\eqref{eq:Hnp_DCT},
write
\[
J_n\pp{\theta} = 2{\rm Re}(W_n(\theta)) \qmwith W_n(\theta):= \summ \ell1n(n-\ell)^2\ell e^{i{\ell\theta} }.
\]
So
\begin{align*}
W_n(\theta) - e^{i\theta}W_n(\theta) & = \summ\ell1n \bb{(n-\ell)^2\ell - (n-\ell+1)^2(\ell-1)}e^{i{\ell\theta}}\\
& = \summ\ell1n\summ{p_1}02\sum_{p_2=0}^{2-p_1}c_{p_1,p_2}n^{p_1}\ell^{p_2}e^{i\ell\theta},
\end{align*}
for some constants $c_{p_1,p_2}$ independent of $n$ and $\theta$. Write $m = \floor{nt} \le n$. Then,
\begin{align}
\abs{J_m\pp{\frac\theta n}}
& \le \frac2{|1-e^{i\theta/n}|}\abs{\summ\ell1{m}\summ {p_1}02\sum_{p_2=0}^{2-p_1}c_{p_1,p_2}m^{p_1}\ell^{p_2}e^{i\ell\theta/n}}
\nonumber\\
& \le C\frac n{|\theta|}\summ {p_1}02m^{p_1}\sum_{p_2=0}^{2-p_1}\abs{\summ\ell1{m}\ell^{p_2}e^{i\ell\theta/n}}= C\frac n{|\theta|}\summ{p_1}02m^{p_1}\sum_{p_2=0}^{2-p_1}\abs{V_{m,p_2+1}\pp{\frac\theta n}}\label{eq:H}
\end{align}
with
\[
V_{n,k}(\theta):= \summ\ell1n\ell^{k-1}e^{i\ell\theta}.
\]
Similarly as above, we have
\[
V_{n,k+1}(\theta) = \frac1{1-e^{i\theta}}\pp{\summ\ell1n\summ j0{k-1}c_{j,k}\ell^je^{i\ell\theta}-n^ke^{i(n+1)\theta}}
\]
for some constants $c_{j,k}$.
So
\[%\begin{equation}\label{eq:Hbound}
\abs{V_{m,k+1}\pp{\frac\theta n}}
 \le \frac {Cn}{|\theta|}\pp{\summ {j}0{k-1}\abs{\summ\ell1m\ell^je^{i\ell\theta/n}}+m^k} = \frac {Cn}{|\theta|}\pp{\summ j1k\abs{V_{m,j}\pp{\frac\theta n}}+m^k}.
\]%\end{equation}
At the same time, $|V_{m,k+1}(\theta/n)|\le m^{k+1}$.
We have seen in~\eqref{eq:D} that, for $m = \floor{nt}\le n$,
\[
\abs{V_{m,1}\pp{\frac\theta n}} \le n\pi \min\ccbb{\frac mn,\frac1{|\theta|}}.
\]
So by induction,
we arrive at
$$
\abs{V_{m,k+1}\pp{\frac\theta n}} \le C_knm^{k}\min\ccbb{\frac mn,\frac1{|\theta|}}, k\in\N,
$$
where $C_k$ are constants depending on $k$. Hence by taking the maximum among $(C_k)_{k=0,1,2}$, we have
$$
\abs{V_{m,k+1}\pp{\frac\theta n}} \le Cnm^{k}\min\ccbb{\frac mn,\frac1{|\theta|}}, k=0,1,2.
$$
Applying this to~\eqref{eq:H} leads to
\begin{equation}\label{eq:boundH}
\abs{J_m\pp{\frac\theta n}} \le C\frac n{|\theta|}\sum_{p_1=0}^2m^{p_1}\sum_{p_2=0}^{2-p_1}(nm^{p_2})\min\ccbb{\frac mn,\frac1{|\theta|}}\le C\frac{n^2m^2}{|\theta|}\min\ccbb{\frac mn,\frac1{|\theta|}}.
\end{equation}
Note also that  $|W_m(\theta)|\le m^4$. We have thus
 proved~\eqref{eq:Hnp_DCT}. Then,
 \[
J_{\sfloor{\vvn\cdot\vvt}}^*(\vv\theta/\vvn) \le C{|\vvn|^4}(t_1t_2)^2\prodd k12\min\ccbb{t_k^2,\frac1{|\theta_k|^2}}.
 \]
Applying this and
\eqref{eq:bound_rhat}
to~\eqref{eq:4th_1},
one has
\begin{align*}
\esp S_\vvn(\vv t)^4 &
\le C |\vvn|^3(t_1t_2)^2 \int\prodd k12 \min\ccbb{t_k^2,\frac1{|\theta_k|^2}}
\\
& \quad \quad \times \pp{\frac{|\vv n|}{n^*}|\theta_1\theta_2|^{1/p(\vv\alpha)-1} + \frac{n_1}{n^*}|\theta_1|^{\alpha_1-1} +\frac{n_2}{n^*}|\theta_2|^{\alpha_2-1} + 1}d\vv\theta
\\
& \le C{|\vvn|^2}(t_1t_2)^2 \cdot \frac{|\vvn|^2}{n^*}|t_1t_2|^{2-1/p(\vv\alpha)},
\end{align*}
where the upper bound for the integral has been treated as in \eqref{eq:Snt2ndM}.
The desired result now follows. \end{proof}
\begin{proof}[Proof of Theorem~\ref{thm:2}]
The proof follows the same line of the proof of Theorem~\ref{thm:1}.
First we establish the finite-dimensional convergence by applying Lindeberg--Feller central limit theorem. The asymptotic covariance of the aggregated random field is the same as the asymptotic covariance of a single random field, up to appropriate normalization, since
\[
\cov(\what S_\vvn(\vvs),\what S_\vvn(\vvt)) = m(\vvn)\cov(S_\vvn(\vvs),S_\vvn(\vvt)).
\]
The latter is established in Proposition~\ref{prop:cov_I_AD}. It remains to verify the counterpart here of the Lindeberg--Feller condition~\eqref{eq:LF2}, which requires the fourth moment on $\esp S_\vvn(\vv t)^4$ established in Proposition~\ref{prop:typeI_4th}. In this case we have, for
$$
Y_\vvn:= \frac1{|\vvn|(n^*)^{-1/2}} \summ w1d a_wS_\vvn(\vvt_w),
$$
\begin{align*}
\esp \pp{Y_\vvn^2\inddd{Y_\vvn^2>m(\vvn)\eta}} & \leq \frac{\esp Y_\vvn^4}{m(\vvn)\eta}\le
 \frac1{m(\vvn)\eta}\pp{\frac{(n^*)^{1/2}}{|\vvn|}\summ w1d |a_w| \pp{\esp S_\vvn(\vvt_w)^4}^{1/4}}^4\\
 & \leq \frac {C(n^*)^2}{m(\vvn)\eta |\vvn|^4}\summ w1d|a_w| \frac{|\vvn|^4}{n^*}= \frac C\eta\frac {n^*}{m(\vvn)},\\
\end{align*}
which converges to 0 as $\vvn\to\infty$ for all $\eta>0$ under condition \eqref{eq:m(n)_I_AD}.

The tightness follows from~\eqref{eq:bickel}, which implies the condition \eqref{eq:Bickel} introduced by \citet{bickel71convergence}.
\end{proof}

\section{Proof of Theorem \ref{thm:non_critical}}\label{sec:non_critical}
We start by explaining how to identify the limits of each regime of non-critical speed, and the corresponding orders of the normalizations. Taken such information for granted, one could  prove Theorem~\ref{thm:non_critical} directly by starting  from
the first section of the Supplementary Material. However, the identification of the
four regimes (essentially two due to symmetry) are at the core of the problem, and we explain this step first.
We also discuss the boundary case in
the last section of the Supplementary Material.

Again we start with computing the asymptotic covariance, which shall indicate the  normalization order and the limit Gaussian random field in each regime.
We still apply the Fourier transform, and Lemma \ref{lem:covariance} still holds:
\[
\cov(S_{\vvn}(\vvs),S_{\vvn}(\vvt))
 = \frac1{(2\pi)^2}\int_{(-\pi,\pi)^2}D_{\vvn,\vvs,\vvt}(\vv\theta)\,\what r\pp{\vv\theta}d\vv\theta.
\]
The evaluation of the asymptotics of the covariance in general depends on two changes of variables.
First, introduce change of variables
\[
\vv\theta\to \frac{\vv\theta}{\vvn'}:=\pp{\frac{\theta_1}{n_1'},\frac{\theta_2}{n_2'}} \qmwith \vvn' = (n_1',n_2').
\]
We have taken $\vvn' = \vvn$ in the regime of critical speed. Here, however, we may need to pick $\vvn'$ differently. So our starting point of analysis is the following expression of the covariance function of the random field:
\equh
\cov(S_{\vvn}(\vvs),S_{\vvn}(\vvt))
 = \frac{|\vvn'|\inv}{(2\pi)^2}\int_{\vvn'\cdot(-\pi,\pi)^2}D_{\vvn,\vvs,\vvt}(\vv\theta/\vvn')\,\what r\spp{\vv\theta/\vvn'}d\vv\theta.\label{eq:cov1'}
\eque
Next, we take a closer look at $\what r(\vv\theta/\vvn')$.
Recall
\[ g(u,\theta) = \frac{u(2-u)}{u^2+2(1-u)(1-\cos\theta)}.
\]
Then, we have, for $\vv\theta\in \vvn'\cdot(-\pi,\pi)^2$,
\begin{align}\label{eq:rhat}
\what r(\vv\theta/\vvn')  & = \int_{\Delta_1}\int_0^\infty \prodd k12 g\pp{(rw_k)^{-1/\alpha_k}, \theta_k/n'_k}\indd{r\vvw\in T_{\vv\alpha}((0,1]^2)}\frac{dr}{r^2}\Lambda(d\vvw)\\
& =
\frac 1{n^*}\int_{\Delta_1}\int_0^\infty \prodd k12 g\pp{(n^*rw_k)^{-1/\alpha_k}, {\theta_k}/{n'_k}}\indd{n^*r\vvw\in T_{\vv\alpha}((0,1]^2)}\frac{dr}{r^2}\Lambda(d\vvw),\nonumber
\end{align}
where $n^{*}$ is a scalar factor satisfying $n^*\to\infty$ as $\vvn\to\infty$
(the rate to be discussed later),
 and the last step follows by the change of variables
\[r\to n^* r.
\] So we can write the integral in~\eqref{eq:cov1'} as a multiple integral over $\R^2\times\R_+\times\Delta_1$ with respect to the measure $d\vv\theta r^{-2}dr\Lambda(d\vvw)$, with the integrand
\begin{multline*}%\label{eq:integrand}
(n^*)^{-1}\prodd k12 D_{\sfloor{n_ks_k}}\pp{\frac{\theta_k}{n'_k}}\wb{D_{\sfloor{n_kt_k}}\pp{\frac{\theta_k}{n'_k}}}\\
\times \prodd k12g\pp{(n^*rw_k)^{-1/\alpha_k}, \frac{\theta_k}{n'_k}}\indd{n^*r\vvw\in T_{\vv\alpha}((0,1]^2)}\inddd{\vv\theta\in\vvn'\cdot(-\pi,\pi)^2}.
\end{multline*}
As before, pointwise asymptotics of $D$ and $g$ are straightforward. We have
\begin{multline}\label{eq:limit_D}
\limn (n_k)^{-2}D_{\sfloor{n_ks_k}}\pp{\frac{\theta_k}{n'_k}}\wb{D_{\sfloor{n_kt_k}}\pp{\frac{\theta_k}{n'_k}}} \\
=\mathfrak D_{s_k,t_k}(\theta_k):=\begin{cases}
\displaystyle \frac{(e^{is_k\theta_k}-1)\wb{(e^{it_k\theta_k}-1)}}{|\theta_k|^2} & n'_k\sim n_k\\
\displaystyle s_kt_k &  n_k' \gg n_k,
\end{cases}
\end{multline}
and
\equh\label{eq:limit_g}
\limn \frac{ g\pp{{(n^*r)^{-1/\alpha_k}}, {\theta_k}/{n'_k}} }{(n^*)^{1/\alpha_k}}
= \mathfrak g(r^{-1/\alpha_k},\theta_k) := \begin{cases}
\displaystyle\frac{2r^{-1/\alpha_k}}{r^{-2/\alpha_k}+\theta_k^2} & (n^*)^{1/\alpha_k} \sim n'_k\\
\displaystyle2r^{1/\alpha_k} & (n^*)^{1/\alpha_k} \ll n'_k.
\end{cases}
\eque
We shall choose $n_1', n_2'$ and $n^*$ as functions of $n_1$ or $n_2$.
In this way, combining~\eqref{eq:cov1'}, \eqref{eq:rhat}, \eqref{eq:limit_D} and~\eqref{eq:limit_g}, we have, {\em formally},
\begin{multline}
\lim_{\vvn\to\infty} \frac{|\vvn'|\cov(S_{\vvn}(\vvs),S_{\vvn}(\vvt))}{|\vvn|^2(n^*)^{p(\vv\alpha)-1}} \\
=  \frac1{(2\pi)^2} \int_{\Delta_1}\int_0^\infty\int_{\R^2}\prodd k12 \mathfrak D_{s_k,t_k}(\theta_k)\mathfrak g((rw_k)^{-1/\alpha_k},\theta_k)d\vv\theta \frac{dr}{r^2}\Lambda(d\vvw),\label{eq:cov3}
\end{multline}
where the functions $\mathfrak D$ and $\mathfrak g$ depend on the choice of $\vvn'$ and $n^*$, and we only computed the pointwise convergence of the multiple integral.

However, a careful examination shall tell quickly that not all choices of $\vvn'$ and $n^*$ will make~\eqref{eq:cov3} a legitimate statement, as the multiple integral is not always well defined: so we need those such that  the multiple integral in~\eqref{eq:cov3} is well defined, finite and strictly non-zero.
The first natural case to be considered is when both $\mathfrak D$ and $\mathfrak g$ are not degenerate, corresponding to the regime of critical speed already addressed in Theorem \ref{thm:2}, with
\[
\vvn' = \vvn, n^*\sim n_1^{\alpha_1}\sim n_2^{\alpha_2}.
\]
Then, it is not hard to see that the only other legitimate integrands  are
\begin{align}%\label{eq:integrand_D1}
& \prodd k12 \frac{(e^{is_k\theta_k}-1)\wb{(e^{it_k\theta_k}-1)}}{\theta_k^2} \frac{2(rw_1)^{-1/\alpha_1}}{(rw_1)^{-2/\alpha_1}+\theta_1^2}2(rw_2)^{1/\alpha_2},\nonumber\\
\label{eq:integrand_D2}
& \prodd k12 \frac{(e^{is_k\theta_k}-1)\wb{(e^{it_k\theta_k}-1)}}{\theta_k^2}2(rw_1)^{1/\alpha_1} \frac{2(rw_2)^{-1/\alpha_2}}{(rw_2)^{-2/\alpha_2}+\theta_2^2},
%YZ0507 added comma, here and below
\\
\label{eq:integrand_g1}
& \frac{(e^{is_1\theta_1}-1)\wb{(e^{it_1\theta_1}-1)}}{\theta_1^2}s_2t_2\prodd k12 \frac{2(rw_k)^{-1/\alpha_k}}{(rw_k)^{-2/\alpha_k}+\theta_k^2},\\
%\label{eq:integrand_g2}
\nonumber
& s_1t_1 \frac{(e^{is_2\theta_2}-1)\wb{(e^{it_2\theta_2}-1)}}{\theta_2^2}\prodd k12 \frac{2(rw_k)^{-1/\alpha_k}}{(rw_k)^{-2/\alpha_k}+\theta_k^2},
\end{align}
and they correspond to the following four conditions on $\vvn'$ and $n^*$, respectively
\begin{align}
\vvn' & = \vvn, n^*\sim n_1^{\alpha_1}, n_1^{\alpha_1}\ll n_2^{\alpha_2},\nonumber\\
\vvn' & = \vvn, n^*\sim n_2^{\alpha_2}, n_1^{\alpha_1}\gg n_2^{\alpha_2},\label{eq:D2}\\
\vvn' & \sim \spp{(n^*)^{1/\alpha_1},(n^*)^{1/\alpha_2}}, n^*\sim n_1^{\alpha_1}, n_1^{\alpha_1}\gg n_2^{\alpha_2},\nonumber\\%\label{eq:g1}\\
\vvn' & \sim \spp{(n^*)^{1/\alpha_1},(n^*)^{1/\alpha_2}}, n^*\sim n_2^{\alpha_2}, n_1^{\alpha_1}\ll n_2^{\alpha_2}.\nonumber
\end{align}
We shall also see later that, for each integrand above to be integrable, an extra assumption on $\vv\alpha$ is needed.

By symmetry, it suffices to focus on the case
\[
n_1^{\alpha_1}\gg n_2^{\alpha_2},
\]
which from now on we assume.
Two identities are needed in these regimes with non-critical speed. The first identity is on the covariance function of fractional Brownian motion (e.g.~\citep[Proposition 7.2.8]{samorodnitsky94stable})
\equh\label{eq:cov_fBm}
\int_\R\frac{(e^{is\theta}-1)\wb{(e^{it\theta}-1)}}{|\theta|^{1+2H}}d\theta = {2\pi}C_H\cov(\B^H_s,\B^H_t), s,t>0,H\in(0,1)
\eque
with
\[
C_H = \frac\pi{H\Gamma(2H)\sin(H\pi)}.
\]
The second is the following
\begin{multline}\label{eq:integrate_r}
\int_0^\infty \frac{r^{-\gamma}}{(rw)^{-2/\alpha}+\theta^2} dr = \frac\alpha2B\pp{H-\frac12,\frac32-H}\frac{w^{\gamma-1}}{|\theta|^{2H-1}} \\
\mbox{ if } H := \frac{3-\alpha(\gamma-1)}2 \in(1/2,3/2),
\end{multline}
and otherwise the integral is infinite. Here $B(a,b) = \int_0^1 x^{a-1}(1-x)^{b-1}dx$ is the beta function.
Indeed,  by change of variables, we have
\begin{align*}
\int_0^\infty\frac{r^{-\gamma}}{(rw)^{-2/\alpha}+\theta^2}dr &= w^{\gamma-1}|\theta|^{\alpha(\gamma-1)-2}\int_0^\infty\frac{r^{-\gamma}}{r^{-2/\alpha}+1}dr \\
& = \frac{w^{\gamma-1}}{|\theta|^{2-\alpha(\gamma-1)}} \frac\alpha 2\int_0^\infty \frac{r^{-(1+(1-\gamma)\alpha/2)}}{r+1}dr.
\end{align*}
Recall also that
$\int_0^\infty (1+u)^{-1}u^{-\beta}du = B(\beta,1-\beta) = \pi/\sin(\pi\beta)$ for all $\beta\in(0,1)$,
and otherwise the integral is infinite. Combining the above yields \eqref{eq:integrate_r}.

We begin with the case \eqref{eq:D2}, by  formally integrating~\eqref{eq:integrand_D2} with respect to $d\vv\theta r^{-2}dr\Lambda(d\vvw)$.
First, by~\eqref{eq:cov_fBm},
\[
\int_\R\frac{(e^{is_1\theta_1}-1)\wb{(e^{it_1\theta_1}-1)}}{|\theta_1|^2}d\theta_1 = 2\pi\cov(\B^{1/2}_{s_1},\B^{1/2}_{t_1}).
\]
Next, by~\eqref{eq:integrate_r},
\equh\label{eq:r2}
\int_0^\infty \frac{(rw_1)^{1/\alpha_1}(rw_2)^{-1/\alpha_2}}{(rw_2)^{-2/\alpha_2}+\theta_2^2}\frac{dr}{r^2} = \frac{\alpha_2}2B\pp{H_2-\frac 12, \frac32 - H_2}\frac{w_1^{1/\alpha_1}w_2^{1-1/\alpha_1}}{|\theta_2|^{2H_2-1}},
\eque
with
\equh\label{eq:case_D,H_2}
H_2 = 1-\frac{\alpha_2}2\pp{1-\frac 1{\alpha_1}} \quad\mbox{ provided }\quad
\frac{\alpha_2}2\pp{1-\frac1{\alpha_1}}\in(0,1/2).
\eque
So the above formal calculation yields an extra necessary assumption $\alpha_1>1$ for the case \eqref{eq:D2}, and in this case integrating~\eqref{eq:integrand_D2} with respect to $d\vv\theta r^{-2}dr\Lambda(d\vvw)$ yields, with $H_1 = 1/2$ and $H_2$ as in~\eqref{eq:case_D,H_2},
\begin{align*}
(2\pi)^2 & \int_{\Delta_1}w_1^{1/\alpha_1}w_2^{1-1/\alpha_1} \Lambda(d\vvw)2{\alpha_2}B\pp{H_2-\frac 12, \frac32 - H_2}C_{H_2}\cov(\B^{\vvH}_{\vvs},\B^{\vvH}_{\vvt})\\
& = (2\pi)^22\alpha_2\mathfrak c_{H_2}\int_{\Delta_1}w_1^{1/\alpha_1}w_2^{1-1/\alpha_1}\Lambda(d\vvw)\cov(\B_\vvs^\vvH,\B_\vvt^\vvH)  = (2\pi)^2\sigma^2\cov(\B^\vvH_\vvs,\B^\vvH_\vvt),
\end{align*}
with $\sigma$ as in regime \eqref{case:D,beta<1} in Theorem \ref{thm:non_critical}.

Now we identify the regime \eqref{case:g,beta=1}.  This time, the multiple integral on the right-hand side of~\eqref{eq:cov3} becomes (by integrating~\eqref{eq:integrand_g1})
\begin{multline*}
\int_{\Delta_1} \int_0^\infty\int_{\R^2}  \frac{(e^{is_1\theta_1}-1)\wb{(e^{it_1\theta_1}-1)}}{|\theta_1|^2}s_2t_2\prodd k12\frac{2(rw_k)^{-1/\alpha_k}}{(rw_k)^{-2/\alpha_k}+\theta_k^2}\frac{dr}{r^2}d\vv\theta \Lambda(d\vvw)\\
  = 2\pi(s_2t_2)\int_\R \frac{(e^{is_1\theta_1}-1)\wb{(e^{it_1\theta_1}-1)}}{|\theta_1|^2} \int_{\Delta_1}\int_\R
 \frac{2(rw_1)^{-1/\alpha_1}}{(rw_1)^{-2/\alpha_1}+\theta_1^2}\frac{dr}{r^2} \Lambda(d\vvw)d\theta_1.
\end{multline*}
Again by~\eqref{eq:integrate_r}, for
\[
H_1 = 1-\frac{\alpha_1}2, H_2 = 1 \quad\mbox{ provided }\quad \alpha_1\in(0,1),
\]
the above becomes
\begin{multline*}
2\pi(s_2t_2)\alpha_1B\pp{H_1-\frac 12,\frac32-H_1} \int_{\Delta_1}w_1\Lambda(d\vvw) \int_\R
\frac{(e^{is_1\theta_1}-1)\wb{(e^{it_1\theta_1}-1)}}{|\theta_1|^{1+2H_1}}d\theta_1\\
= (2\pi)^2\alpha_1\mathfrak c_{H_1}\int_{\Delta_1}w_1\Lambda(d\vvw) \cov(\B_\vvs^\vvH,\B_\vvt^\vvH).
\end{multline*}
This is the regime~\eqref{case:g,beta=1}.

To complete the computation of asymptotic covariance~\eqref{eq:cov3},
it remains to provide an integrable bound to apply the dominated convergence theorem. To establish the limit theorem, we need to also bound the fourth-moment. These are left to
the first two sections in the Supplementary Material.

\begin{supplement} [id=suppA]
%\sname{Supplement A}
\stitle{Proof of Theorem~\ref{thm:non_critical} and the boundary case}
\slink[doi]{COMPLETED BY THE TYPESETTER}
\sdatatype{.pdf}
\sdescription{We prove the two regimes in Theorem~\ref{thm:non_critical} and the convergence of  covariance in the boundary case for non-critical speed.}
\end{supplement}

\subsection*{Acknowledgement}
%YZ0507 added
The authors thank two anonymous referees for careful reading and helpful comments. 
YS's research was supported in part by the Natural Sciences and Engineering Research Council of Canada (RGPIN-2014-04840).
YW's research was supported in part by NSA grant H98230-16-1-0322 and Army Research Laboratory grant W911NF-17-1-0006.

 \bibliographystyle{apalike}
%\bibliography{../include/references}
\bibliography{references}

%%%%%%%%%%%%%%%%%

\begin{frontmatter}

% "Title of the Paper"
%\title{???}

%\runtitle{???}

\title{Supplement to ``Operator-scaling Gaussian random fields via aggregation''}

\runtitle{Supplement to ``Operator-scaling Gaussian random fields via aggregation''}

\begin{aug}
% indicate corresponding author with \corref{}
% \author{\fnms{John} \snm{Smith}\thanksref{a}\corref{}\ead[label=e1]{smith@foo.com}\ead[label=e2,url]{www.foo.com}}
% \address[a]{\printead{e1};\printead{e2}}

\author{\fnms{Yi} \snm{Shen}\thanksref{aa}\ead[label=ee1]{yi.shen@uwaterloo.ca}}
\and
\author{\fnms{Yizao} \snm{Wang}\thanksref{bb}\corref{}\ead[label=ee2]{yizao.wang@uc.edu}}
\address[aa]{Department of Statistics and Actuarial Science,
University of Waterloo,
Mathematics 3 Building,
200 University Avenue West
Waterloo, Ontario N2L 3G1,
Canada.
\printead{ee1}}
\address[bb]{
Department of Mathematical Sciences,
University of Cincinnati,
2815 Commons Way, ML--0025,
Cincinnati, OH, 45221-0025, USA.
\printead{ee2}}

\runauthor{Shen and Wang}

\affiliation{University of Waterloo and University of Cincinnati}

\end{aug}

% history:
% \received{\smonth{1} \syear{0000}}

%\tableofcontents

\end{frontmatter}

We provide a proof of \citep[Theorem 1.5]{shen17operator} in Sections \ref{sec:proof_non_critical} and \ref{sec:3_integral} below. The boundary case of the non-critical regime is investigated in Section \ref{sec:boundary}.

\section{Proof of Theorem~\ref{thm:non_critical}, regime~\eqref{case:D,beta<1}}
\label{sec:proof_non_critical}
%Consider regime~\eqref{case:D,beta<1}.
Recall that in this regime, we have
\equh\label{eq:D2'}
\vvn' =\vvn, n^* = n_2^{\alpha_2} \qmwith n_1^{\alpha_1}\gg n_2^{\alpha_2}, \alpha_1>1,
\eque
which we assume throughout this section without further mentioning.
Again we start by computing the second moments.
We write
\[
\cov\pp{S_\vvn(\vvs),S_\vvn(\vvt)} = \frac{|\vvn|\inv}{(2\pi)^2}\int_{\vvn\cdot(-\pi,\pi)^2}D_{\vvn,\vvs,\vvt}(\vv\theta/
\vvn)\what r(\vv\theta/\vvn)d\vv\theta.
\]
We have seen how to estimate and control $D$ in Lemma~\ref{lem:D}. We need the following estimates on $\what r(\vv\theta/\vvn)$.
\begin{Prop}\label{prop:rhat_theta_n}
Under assumption \eqref{eq:D2'},
\[
\what r(\vv\theta/\vvn) \sim \frac{n_2^{2H_2-1}}{|\theta_2|^{2H_2-1}}\int_{\Delta_1}w_1^{1/\alpha_1}w_2^{1-1/\alpha_1}\Lambda(d\vvw)2\alpha_2B\pp{H_2-\frac12,\frac32-H_2},
\]
and
\[
\what r(\vv\theta/\vvn) \le C\pp{\frac{n_2^{2H_2-1}}{|\theta_2|^{2H_2-1}}+1}.
\]
\end{Prop}
\begin{proof}
In this case, \eqref{eq:rhat} becomes
\equh\label{eq:rhat1}
\what r(\vv\theta/\vvn) = (n^*)\inv\int_{\Delta_1}\int_0^\infty \prodd k12g\pp{(n^*rw_k)^{-1/\alpha_k},{\theta_k}/{n_k}}\indd{n^*r\vvw\in T_{\vv\alpha}((0,1]^2)}\frac{dr}{r^2}\Lambda(d\vvw).
\eque

We start by proving the second part of the proposition.
For the first direction (corresponding to $H_1 = 1/2$), we use the simple bound $g(u,\theta)\le 2u\inv$. For the second direction (corresponding to $H_2\in(1/2,1)$), again we break the integral into two parts. Introduce and fix $\delta\in(0,1)$. As in~\eqref{eq:<delta}, we have
\[
g\pp{(n^*rw_2)^{-1/\alpha_2},{\theta_2}/{n_2}}\inddd{(n^*rw_2)^{-1/\alpha_2}\in(0,\delta]}\le \frac{C(n^*)^{1/\alpha_2}}{(rw_2)^{1/\alpha_2}\theta_2^2}.
\]
Then,
\begin{align}
\int_0^\infty & \prodd k12 g\pp{(n^*rw_k)^{-1/\alpha_k},{\theta_k}/{n_k}}\inddd{(n^*rw_2)^{-1/\alpha_2}\in(0,\delta]}\frac{dr}{r^2} \nonumber\\
& \le C(n^*)^{p(\vv\alpha)}\int_0^\infty r^{1/\alpha_1-2}w_1^{1/\alpha_1}\min\ccbb{(rw_2)^{1/\alpha_2},\frac1{(rw_2)^{1/\alpha_2}\theta_2^2}}dr\nonumber\\
& \le C(n^*)^{p(\vv\alpha)}\frac{w_1^{1/\alpha_1}w_2^{1-1/\alpha_1}}{|\theta_2|^{2H_2-1}},\label{eq:rhat1_1}
\end{align}
where in the last step again we break the integral into two parts at $(rw_2)^{1/\alpha_2} = 1/|\theta_2|$, and recall that in this case $H_2$ is given in~\eqref{eq:case_D,H_2}. At the same time,
\begin{align}
\int_0^\infty  \prodd k12  & g\pp{(n^*rw_k)^{-1/\alpha_k},{\theta_k}/{n_k}}\inddd{(n^*rw_2)^{-1/\alpha_2}\in(\delta,1]}\frac{dr}{r^2} \nonumber\\
& = \int_{(n^*w_2)\inv}^{(n^*w_2)\inv \delta^{-\alpha_2}} \prodd k12 g\pp{(n^*rw_k)^{-1/\alpha_k},{\theta_k}/{n_k}}\frac{dr}{r^2} \nonumber\\
& \le C(n^*)^{p(\vv\alpha)}w_1^{1/\alpha_1}w_2^{1/\alpha_2}\int_{(n^*w_2)\inv}^{(n^*w_2)\inv \delta^{-\alpha_2}} r^{p(\vv\alpha)-2}dr  = Cn^*w_1^{1/\alpha_1}w_2^{1-1/\alpha_1},\label{eq:rhat1_2}
\end{align}
where we used the bound $g(u,\theta)\le 2u\inv$ twice. Therefore, combining~\eqref{eq:rhat1},~\eqref{eq:rhat1_1} and~\eqref{eq:rhat1_2}, we have
\[
\what r(\vv\theta/\vvn)\le \int_{\Delta_1}w_1^{1/\alpha_1}w_2^{1-1/\alpha_1}\Lambda(d\vvw)\pp{(n^*)^{p(\vv\alpha)-1}|\theta_2|^{1-2H_2}+1}.
\]
Remark that $(n^*)^{p(\vv\alpha)-1} = n_2^{1+\alpha_2/\alpha_1-\alpha_2} = n_2^{2H_2-1}$ by \eqref{eq:case_D,H_2} and \eqref{eq:D2'}.
This proves the second part of the proposition. For the first part, to show the asymptotics of \eqref{eq:rhat1}, we have seen before
\[
(n^*)^{-p(\vv\alpha)} \prodd k12 g\pp{(n^*rw_k)^{-1/\alpha_k},{\theta_k}/{n_k}} \sim
2(rw_1)^{1/\alpha_1}
\frac{2(rw_2)^{-1/\alpha_2}}{(rw_2)^{-2/\alpha_2}+\theta_2^2}.
\]
So by the dominated convergence theorem and \eqref{eq:integrate_r} (see \eqref{eq:r2}), the first part of the proposition follows. We omit the details.
\end{proof}

\begin{Prop}
Under assumption \eqref{eq:D2'},
\[
\lim_{\vvn\to\infty}
 \frac{\cov(S_{\vvn}(\vvs),S_{\vvn}(\vvt))}{n_1^{2H_1}n_2^{2H_2}} = \sigma^2\cov(\B^\vvH_\vvs,\B^\vvH_\vvt),
\]
with $\vvH,\sigma$ as in regime~\eqref{case:D,beta<1} in Theorem~\ref{thm:non_critical}.
\end{Prop}
\begin{proof}
By the estimates on $\what r(\vv\theta/\vvn)$ and $D$ in Proposition~\ref{prop:rhat_theta_n} and Lemma~\ref{lem:D} respectively, we have
\begin{align*}
 \frac{\cov(S_\vvn(\vvs),S_\vvn(\vvt))}{n_1^{2H_1}n_2^{2H_2}}  &= \frac{|\vvn|\inv}{n_1n_2^{2H_2}(2\pi)^2}\int_{\vvn\cdot(-\pi,\pi)^2}D_{\vvn,\vvs,\vvt}(\vv\theta/\vvn)\what r(\vv\theta/\vvn)d\vv\theta\\
& \sim \frac1{(2\pi)^2}\int_{\R^2}\prodd k12 \frac{(e^{is_k\theta_k}-1)\wb{(e^{it_k\theta_k}-1)}}{|\theta_k|^2}\frac1{|\theta_2|^{2H_2-1}}d\vv\theta\\
& \quad\times \int_{\Delta_1}w_1^{1/\alpha_1}w_2^{1-1/\alpha_1}\Lambda(d\vvw)2\alpha_2B\pp{H_2-\frac 12,\frac32-H_2}\\
& = \sigma^2\cov(\B_\vvs^\vvH,\B_\vvt^\vvH),
\end{align*}
by the dominated convergence theorem.
\end{proof}
\begin{Prop}For $\vvt$ such that $\floor{\vvn\cdot\vvt} = \vvn\cdot\vvt$, we have
\[
\esp S_\vvn(\vvt)^4 \le Cn_1^3n_2^{2H_2+2}t_1^3t_2^{2H_2+2}.
\]
\end{Prop}
\begin{proof}
By \eqref{eq:4th_1} and \eqref{eq:Hnp_DCT} and Proposition \ref{prop:rhat_theta_n}, we have
\begin{align*}
\esp S_\vvn(\vvt)^4 &\le C|\vvn|\inv\int_{\vvn\cdot(-\pi,\pi)^2}J_{\floor{\vvn\cdot \vvt}}^*(\vv\theta/\vvn)\what r(\vv\theta/\vvn)d\vv\theta\\
& \le C|\vvn|\inv |\vvn|^4|\vvt|^2\int_{\R^2}\prodd k12\min\ccbb{t_k^2,\frac1{|\theta_k|^2}}(n_2^{2H_2-1}|\theta_2|^{1-2H_2}+1)d\vv\theta\\
& \le C(n_1^3n_2^{2H_2+2}t_1^3t_2^{2H_2+2} + |\vvn|^3|\vvt|^3).
\end{align*}
The desired result now follows, by noticing that $H_2\in(1/2,1)$.
\end{proof}
\begin{proof}[Proof of Theorem \ref{thm:non_critical}, regime \eqref{case:D,beta<1}]
Set $Y_\vvn$ as in \eqref{eq:Yn}.
We only verify the Lindeberg--Feller condition. This time we have
\begin{align*}
\esp\pp{Y_\vvn^2\inddd{Y_\vvn^2>m(\vvn)\eta}} & \leq \frac1{m(\vvn)\eta}\esp Y_\vvn^4 \le
 \frac1{m(\vvn)\eta}\pp{\frac1{n_1^{1/2}n_2^{H_2}}\summ w1d |a_w| \pp{\esp S_\vvn^4(\vvt_w)}^{1/4}}^4 \\
& \leq \frac C{m(\vvn)\eta n_1^2n_2^{4H_2}} n_1^3n_2^{2H_2+2}= \frac C\eta\frac {n_1n_2^{2-2H_2}}{m(\vvn)}.
\end{align*}
This completes the proof.
\end{proof}
%\newpage
\section{Proof of Theorem \ref{thm:non_critical}, regime \eqref{case:g,beta=1}}\label{sec:3_integral}
%Consider regime~\eqref{case:g,beta=1}.
Recall that in this regime,
\equh\label{eq:D3'}
n^* = n_1^{\alpha_1} \qmand \vvn'= (n_1,n_1^{\alpha_1/\alpha_2})   \qmwith \alpha_1<1, n_1^{\alpha_1}\gg n_2^{\alpha_2},
\eque
which we assume throughout without further specification.  The treatment is slightly different from the previous cases in the sense that we have to work with the representation of covariance function as a triple integral (see \eqref{eq:triple0} below), and apply the dominated convergence theorem once for all.

\begin{Prop}
Under the assumption \eqref{eq:D3'}, we have
\[
\lim_{\vvn\to\infty}
\frac{\cov(S_{\vvn}(\vvs),S_{\vvn}(\vvt))}{n_1^{2H_1}n_2^{2H_2}} = \sigma^2\cov(\B^\vvH_\vvs,\B^\vvH_\vvt),
\]
with $\vvH = (1-\alpha_1/2,1)$, and $\sigma$ as in regime~\eqref{case:g,beta=1} in Theorem~\ref{thm:non_critical}.
\end{Prop}

\begin{proof}
This time, to simplify the notation we introduce
\[
\wt h_\vvn(r,\vvw,\vv\theta) := \prodd k12 g\pp{(n^*rw_k)^{-1/\alpha_k}, {\theta_k}/{n'_k}}\indd{n^*r\vvw\in T_{\vv\alpha}((0,1]^2)}.
\]
Combining (\ref{eq:cov1'}) and (\ref{eq:rhat}), we have
\begin{multline}\label{eq:triple0}
\frac{\cov(S_{\vvn}(\vvs),S_{\vvn}(\vvt))}{n_1^{2H_1}n_2^{2H_2}}\\
=  \frac1{|\vvn'||\vvn|^2(2\pi)^2}
\int_{\Delta_1}\int_0^\infty\int_{\vvn'\cdot(-\pi,\pi)^2}D_{\vvn,\vvs,\vvt}(\vv\theta/\vvn')
%\times  %\prodd k12 g\pp{(n^*rw_k)^{-1/\alpha_k}, {\theta_k}/{n'_k}}\indd{n^*r\vvw\in T_{\vv\alpha}((0,1]^2)}
\wt h_\vvn(r,\vvw,\vv\theta) r^{-2}d\vv\theta dr \Lambda(d\vvw).
\end{multline}

We shall apply the dominated convergence theorem. We have seen the pointwise convergence of the normalized integrand %YZ1206 before
in Section \ref{sec:non_critical},
 so it suffices to %YS1206 look for
 find an integrable bound.
Similarly as in the critical case, we divide this integral into two parts with domains determined by
\[
\Omega_\vvn^{0,0} :=\ccbb{n^*r\vvw\in T_{\vv\alpha}((0,\delta]^2)} \qmand \Omega_\vvn := \ccbb{n^*r\vvw\in T_{\vv\alpha}((0,1]^2\setminus(0,\delta]^2)},
\]
respectively. Introduce also  $\Xi_{\vvn'} := \{\vv\theta\in\vvn'\cdot(-\pi,\pi)^2\}$. %YS1205: I would suggest to use \Xi_{\vvn'}, as it is directly and formally a function of \vvn'.%YZ 1205 agreed!
We first justify the application of the dominated convergence theorem to
\begin{equation}\label{integral1}
 \frac1{|\vvn'||\vvn|^2(2\pi)^2}\int_{\Delta_1}\int_0^\infty\int_{\R^2}D_{\vvn,\vvs,\vvt}(\vv\theta/\vvn')
 \ind_{\Xi_{\vvn'}}\wt h_\vvn(r,\vvw,\vv\theta)
\ind_{\Omega_\vvn^{0,0}}
r^{-2}d\vv\theta dr \Lambda(d\vvw).
\end{equation}
Similarly as in Lemma \ref{lem:D}, we have, for $k=1,2$, $|\theta_k|\le n_k'\pi$,
$$
\left|\frac{D_{\lfloor n_kt_k\rfloor}(\theta_k/n_k')}{n_k}\right|=\frac{|\sin(\lfloor n_kt_k\rfloor \theta_k/2n_k')|}{n_k|\sin(\theta_k/2n_k')|} \le \min\left\{\frac {\pi t_k}2,
\frac{n'_k}{|\theta_k|n_k}\right\}.
$$
Hence
\equh\label{eq:D1}
|D_{\vvn,\vvs,\vvt}(\vv\theta/\vvn')|\ind_{\Xi_{\vvn'}}
\leq \pp{\frac \pi2}^4|\vvn|^2\min\ccbb{s_1t_1, \frac1{|\theta_1|^2}}s_2t_2.
\eque
On $\Omega_\vvn^{0,0}$,
using in addition
\begin{align*}
 g\pp{(n^*rw_k)^{-1/\alpha_k}, {\theta_k}/{n'_k}} & \le \frac{2(n^*rw_k)^{1/\alpha_k}}{1+C_\delta (n^*rw_k)^{2/\alpha_k}(1-\cos(\theta_k/(n^*)^{1/\alpha_k})}\\
& \le \frac{2(n^*rw_k)^{1/\alpha_k}}{1+C_\delta \theta_k^2(rw_k)^{2/\alpha_k}},
 \end{align*}
 where we used the fact that $1-\cos\theta\ge 2\theta^2/\pi^2$ for all $\theta\in[-\pi,\pi]$,
 we bound the integrand of \eqref{integral1} (with the normalization $(|\vvn'||\vvn|^2)\inv$) by
 \[
 C \min\left\{s_1t_1, \frac1{\theta_1^2}\right\}s_2t_2 \prod_{k=1}^2\frac{(rw_k)^{1/\alpha_k}}{1+C_\delta (rw_k)^{2/\alpha_k}\theta_k^2}r^{-2}.
 \]
 %YS1205 removed (We keep $s_2t_2$ for a later purpose.) I don't think we need to explain this, as keeping $s_2t_2$ is natural, while removing them is an extra operation. %YZ1205 good to me.
 To see that this is integrable, we integrate it with respect to $d\theta_2,dr$ and $d\theta_1$ in order and %YZ1206 see that
 obtain
\begin{align}
 \int_{\Delta_1} & \int_{\R^2}\int_{\R_+}
\min\left\{s_1t_1, \frac1{\theta_1^2}\right\}s_2t_2 \prod_{k=1}^2\frac{(rw_k)^{1/\alpha_k}}{1+C_\delta (rw_k)^{2/\alpha_k}\theta_k^2}r^{-2}dr
d\vv\theta \Lambda(d\vvw)\nonumber\\
& = C \int_{\Delta_1}\int_{\R}\int_{\R_+} \min\left\{s_1t_1, \frac1{\theta_1^2}\right\}s_2t_2\frac{ (rw_1)^{1/{\alpha_1}}}{1+C_\delta (rw_1)^{2/\alpha_1}\theta_1^2}\frac{dr}{r^2}d\theta_1\Lambda(d\vvw)\nonumber\\
& = C_\delta\int_{\Delta_1}\int_\R \min\ccbb{s_1t_1,\frac1{\theta_1^2}}s_2t_2w_1 |\theta_1|^{\alpha_1-1}d\theta_1\Lambda(d\vvw)\nonumber\\
& = C_\delta(s_1t_1)^{1-{\alpha_1}/{2}}s_2t_2\int_{\Delta_1}w_1 \Lambda(d\vvw)
\leq C_\delta(s_1t_1)^{1-{\alpha_1}/{2}}s_2t_2.\label{eq:J1}
\end{align}

Next,
we show that \eqref{integral1} with $\Omega_\vvn^{0,0}$ replacedy by $\Omega_\vvn$ converges to zero for any given $\delta\in(0,1)$.
Again we divide the region of interest into three pieces,
$\Omega_\vvn = \Omega_\vvn^{0,1}\cup \Omega_\vvn^{1,0}\cup\Omega_\vvn^{1,1}$ with
\begin{align*}
\Omega_\vvn^{0,1} &:= \ccbb{n^*r\vvw\in T_{\vv\alpha}((0,\delta]\times(\delta,1])}\\
\Omega_\vvn^{1,0} &:= \ccbb{n^*r\vvw\in T_{\vv\alpha}((\delta,1]\times(0,\delta])}\\
\Omega_\vvn^{1,1} &:= \ccbb{n^*r\vvw\in T_{\vv\alpha}((\delta,1]^2
)},
\end{align*}
 and on each we apply the inequality
\begin{multline}\label{eq:boundg}
g\pp{(n^*rw_k)^{-1/\alpha_k}, {\theta_k}/{n'_k}}\\
 \le \begin{cases}
\displaystyle C_\delta n'_k \min\ccbb{(rw_k)^{1/\alpha_k},\frac1{(rw_k)^{1/\alpha_k}\theta_k^2}}
& (n^*rw_k)^{-1/\alpha_k}\in (0,\delta]\\ \\
C_\delta & (n^*rw_k)^{-1/\alpha_k}\in (\delta,1]
\end{cases}.
\end{multline}
For $\Omega_\vvn^{1,1}$, by (\ref{eq:D1}) and the above inequality, the integral \eqref{integral1} with $\Omega_\vvn^{0,0}$ replaced by $\Omega_\vvn^{1,1}$
is bounded by
\begin{multline}
 \frac{C_\delta n_2'|\vvn|^2s_2t_2}{|\vvn'||\vvn|^2} \int_\R\min\ccbb{s_1t_1, \frac1{|\theta_1|^2}}d\theta_1\int_{\Delta_1}\int_{(n^*)^{-1}(w_1^{-1}\vee w_2^{-1})}^{(n^*)^{-1}[(w_1\delta^{\alpha_1})^{-1}\vee(w_2\delta^{\alpha_2})^{-1}]}r^{-2} dr\Lambda(d\vvw)\\
 \leq \frac{C_\delta n_2'n^*}{|\vvn'|}(s_1t_1)^{1/2}s_2t_2\int_{\Delta_1}w_1\wedge w_2\Lambda(d\vvw)  %\le \frac{C_\delta n_2'n^*}{|\vvn'|}=C_\delta\frac{n^*}{n_1'}\sim
\le  C_\delta (s_1t_1)^{1/2}s_2t_2 n_1^{\alpha_1-1}\to 0.\label{eq:J2}
\end{multline}
Similarly, %in the area where $\indd{n^*r\vvw\in T_{\vv\alpha}((\delta,1]\times(0,\delta])}$, the integral
\eqref{integral1} with $\Omega_\vvn^{0,0}$ replaced by $\Omega_\vvn^{1,0}$ is bounded by
\begin{multline*}
% \frac1{|\vvn'||\vvn|^2(2\pi)^2}& \int_{\Delta_1}\int_0^\infty\int_{\R^2}D_{\vvn,\vvs,\vvt}(\vv\theta/\vvn')
%\inddd{\vv\theta\in\vvn'\cdot(-\pi,\pi)^2}%\\\times
%\ind_{\Xi_{\vvn'}}\wt h_\vvn(r,\vvw,\vv\theta)
%\prodd k12 g\pp{(n^*rw_k)^{-1/\alpha_k}, {\theta_k}/{n'_k}}
%\indd{n^*r\vvw\in T_{\vv\alpha}((\delta,1]\times(0,\delta])}
%\ind_{\Omega_\vvn^{1,0}}r^{-2}d\vv\theta dr\Lambda(d\vvw)\\
  \frac{C_\delta n_2' s_2t_2}{|\vvn'|}\int_{\Delta_1}
  %\int_{(n^*)^{-1}(w_1^{-1}\vee w_2^{-1})}%^{(n^*)^{-1}[(w_1\delta^{\alpha_1})^{-1}\vee(w_2\delta^{\alpha_2})^{-1}]}
\int_{(n^*w_1)\inv}  ^\infty
   \int_\R\min\ccbb{s_1t_1, \frac1{|\theta_1|^2}}d\theta_1\\
 \int_\R \min\ccbb{(rw_2)^{1/\alpha_2},\frac1{(rw_2)^{1/\alpha_2}\theta_2^2}}d\theta_2 r^{-2} dr\Lambda(d\vvw),
\end{multline*}
which then becomes
\begin{multline}
 \frac{C_\delta }{n_1'}(s_1t_1)^{1/2}s_2t_2\int_{\Delta_1}
\int_{(n^*w_1)\inv}^\infty
 r^{-2} dr\Lambda(d\vvw)\\
\leq  \frac{C_\delta n^*}{n_1'}(s_1t_1)^{1/2}s_2t_2\int_{\Delta_1}w_1
\Lambda(d\vvw)
\leq
 C_\delta (s_1t_1)^{1/2}s_2t_2 n_1^{\alpha_1-1}\to 0.\label{eq:J3}
\end{multline}
The last area $\Omega_\vvn^{0,1}$ requires a more careful treatment. %A bound of the normalized integrand
In this case, \eqref{integral1} with $\Omega_\vvn^{0,0}$ replaced by $\Omega_\vvn^{0,1}$ is bounded by
\[
C_\delta \int_\R\int_{\Delta_1}\int_\R\min\ccbb{s_1t_1, \frac1{|\theta_1|^2}}s_2t_2\frac{1}{n'_1}g\pp{(n^*rw_1)^{-1/\alpha_1}, {\theta_1}/{n'_1}}
%\indd{n^*r\vvw\in T_{\vv\alpha}((0,\delta]\times(\delta,1])}
\ind_{\Omega_\vvn^{0,1}}
r^{-2} dr\Lambda(d\vvw)d\theta_1.
\]
Note that the current integrand
%$$
%\min\ccbb{s_1t_1, \frac1{|\theta_1|^2}}\frac{1}{n'_1}g\pp{(n^*rw_1)^{-1/\alpha_1}, {\theta_1}/{n'_1}}
%\indd{n^*r\vvw\in T_{\vv\alpha}((0,\delta]\times(\delta,1])}
%\ind_{\Omega_\vvn^{0,1}}r^{-2},
%$$
converges pointwisely to 0 as $n^*$ goes to infinity, since $n^*r\vvw$ will eventually leave the area $T_{\vv\alpha}((0,\delta]\times(\delta,1])$. Hence in order to prove the integral converges to 0 as well, it suffices to find an integrable upper bound, and then to apply the dominated convergence theorem. This can be done by applying bound (\ref{eq:boundg}) again:
\begin{multline*}
 \min\ccbb{s_1t_1, \frac1{|\theta_1|^2}}s_2t_2\frac{1}{n'_1}g\pp{(n^*rw_1)^{-1/\alpha_1}, {\theta_1}/{n'_1}}
 %\indd{n^*r\vvw\in T_{\vv\alpha}((0,\delta]\times(\delta,1])}
 \ind_{\Omega_\vvn^{0,1}}
 r^{-2}\\
\leq  C_\delta s_2t_2 \min\ccbb{s_1t_1, \frac1{|\theta_1|^2}}\min\ccbb{(rw_1)^{1/\alpha_1},\frac1{(rw_1)^{1/\alpha_1}\theta_1^2}}r^{-2},
\end{multline*}
which is integrable:
\begin{align}
 \int_\R & \int_{\Delta_1} \int_\R C_\delta  s_2t_2 \min\ccbb{s_1t_1, \frac1{|\theta_1|^2}}\min\ccbb{(rw_1)^{1/\alpha_1},\frac1{(rw_1)^{1/\alpha_1}\theta_1^2}}r^{-2}dr\Lambda(d\vvw)d\theta_1\nonumber\\
& = C_\delta s_2t_2\int_\R\min\ccbb{s_1t_1, \frac1{|\theta_1|^2}}\int_{\Delta_1}\theta_1^{\alpha_1-1}w_1\Lambda(d\vvw)d\theta_1\nonumber\\
& \leq  C_\delta s_2t_2\int_\R\min\ccbb{s_1t_1, \frac1{|\theta_1|^2}}\theta_1^{\alpha_1-1}d\theta_1 = C_\delta s_1^{1-\alpha_1/2}%YS1205 corrected s_2
t_1^{1-\alpha_1/2}s_2t_2.\label{eq:J4}
\end{align}
Therefore, we conclude that \eqref{integral1} with $\Omega_\vvn^{0,0}$ replaced by $\Omega_\vvn$ vanishes asymptotically. This completes the proof.
\end{proof}

It remains to bound the fourth moment.

\begin{Prop}Under the assumption \eqref{eq:D3'}, for $\vvt$ such that $\floor{\vvn\cdot\vvt} = \vvn\cdot\vvt$, we have
\[
\esp S_\vvn(\vvt)^4 \le Cn_1^{2H_1+2}n_2^4t_1^{2H_1+2}t_2^4.
\]
\end{Prop}

\begin{proof}
\begin{align}
\esp S_\vvn(\vvt)^4 &\le C|\vvn'|\inv\int_{\R^2}
%_{\vvn'\cdot(-\pi,\pi)^2}
J_{\floor{\vvn\cdot \vvt}}^*(\vv\theta/\vvn')
\ind_{\Xi_{\vvn'}}
\what r(\vv\theta/\vvn')d\vv\theta\nonumber\\
&=C\frac1{|\vvn'|n^*} \int_{\Delta_1}\int_0^\infty
%\int_{\vvn'\cdot(-\pi,\pi)^2}
\int_{\R^2}J_{\floor{\vvn\cdot \vvt}}^*\left({\vv\theta}/{\vvn'}\right)
\ind_{\Xi_{\vvn'}}
%\prodd k12 g\pp{(n^*rw_k)^{-1/\alpha_k}, {\theta_k}/{n'_k}}
\wt h_\vvn(r,\vvw,\vv\theta)
r^{-2}d\vv\theta dr \Lambda(d\vvw).\label{eq:JJ}
\end{align}
The upper bound of $J$ in \eqref{eq:boundH} (recall that $n_2' = n_1^{\alpha_1/\alpha_2}\gg n_2$) this time yields
\[
J_{\floor{\vv n\cdot \vv t}}^* (\vv\theta/\vv n') \ind_{\Xi_{\vvn'}}\le C|\vvn|^4 t_1^2t_2^4\min\ccbb{t_1^2,\frac1{\theta_1^2}}.
\]
Therefore,
 the integrand of \eqref{eq:JJ} is bounded by
$$
|\vv n|^2(t_1t_2)^2 \cdot  |\vvn|^2\min\left\{t_1^2,\frac{1}{\theta_1^2}\right\}t_2^2 \wt h_\vvn(r,\vvw,\vv\theta)
r^{-2}.
$$
In view of \eqref{eq:D1}, we recognize that one can obtain an upper bound for %YS1205 added
the integral in \eqref{eq:JJ} that differs from the upper bound for %YS1205added
the integral in \eqref{integral1} in the previous case by a multiplicative constant $C_\delta |\vvn|^2(t_1t_2)^2$. In particular, combining \eqref{eq:J1}, \eqref{eq:J2}, \eqref{eq:J3} and \eqref{eq:J4}, we thus arrive at
\begin{align*}
\esp S_\vvn(\vvt)^4 &\le C_\delta \frac1{|\vvn'|n^*}\cdot |\vvn|^2(t_1t_2)^2 \cdot |\vvn'||\vvn|^2 \pp{t_1t_2^2n_1^{\alpha_1-1} + t_1^{2-\alpha_1}t_2^2}\\
& \le C_\delta n_1^{4-\alpha_1}n_2^4(t_1t_2)^2\pp{t_1t_2^2n_1^{\alpha_1-1} + t_1^{2-\alpha_1}t_2^2} \\
& = C_\delta \bb{(n_1t_1)^{4-\alpha_1}(n_2t_2)^4 + (n_1t_1)^3(n_2t_2)^4}.
\end{align*}
By our assumption, $(n_1t_1)^3 \le (n_1t_1)^{4-\alpha_1}$, hence the desired result follows.
\end{proof}
%\newpage
\section{Boundary case of non-critical speed}\label{sec:boundary}
The case for non-critical speed
\equh\label{eq:boundary}
\alpha_1=1 \qmand n_1\gg n_2^{\alpha_2},
\eque
 was not discussed in Theorem \ref{thm:non_critical}. As one would expect by observing the two regimes in Theorem \ref{thm:non_critical}, the case with $\alpha_1=1$ is degenerate and the limit is a fractional Brownian sheet with Hurst indices
\[
\vvH = \pp{1/2, 1}.
\]
 The normalization in the functional limit theorem involves this time, however, a logarithmic term in $\vv n$. We expect to be able to show that
$$
\frac{\what S_\vvn(\vvt)}{\sqrt{n_1\log n_1}n_2\sqrt{m(\vvn)}}\weakto \sigma\ccbb{\B_\vvt^{\vvH}}_{\vvt\in[0,1]^2},
$$
with
\[
\sigma^2 = 4\pi \int_{\Delta_1}w_1\Lambda(d\vvw).
\] Here, we only prove the limit theorem for the covariance function of a single random field.

\begin{Prop}
\label{prop:cov}
With \eqref{eq:boundary}, if
\equh\label{eq:Lambda_boundary}
\int_{\Delta_1}
|\log w_2|\Lambda(d\vvw)<\infty,
\eque
then,
\[
\lim_{\vvn\to\infty} \frac{\cov(S_{\vvn}(\vvs),S_{\vvn}(\vvt))}{n_1\log(n_1)n_2^2}  = \sigma^2\cov(\mathbb B_{\vvs}^{\vvH}, \mathbb B_{\vvt}^{\vvH}).
\]
\end{Prop}

\begin{proof}
Under the assumption \eqref{eq:boundary}, we take
\[
\vvn' = (n_1,n_1^{1/\alpha_2}), n^* = n_1.
\]
%Recall that we have seen in \eqref{eq:cov1'}
%\[
%\cov(S_\vvn(\vvs),S_\vvn(\vvt)) = \frac{|\vvn'|\inv}{(2\pi)^2}\int_{\vvn'\cdot(-\pi,\pi)^2}D_{\vvn,\vvs,\vvt}(\vv\theta/\vvn')\what r(\vv\theta/\vvn')d\vv\theta.
%\]
We still need to work with triple integral representation of the covariance function
\begin{multline}\label{eq:triple}
\frac{\cov(S_{\vvn}(\vvs),S_{\vvn}(\vvt))}{n_1\log (n_1)n_2^2}  =  \frac1{(2\pi)^2\log n_1}
 \int_{\Delta_1}\int_0^\infty\int_{\vvn'\cdot(-\pi,\pi)^2}
 \inddd{n^*r\vvw\in T_{\vv\alpha}((0,1]^2)}
 \\
 \times \frac{D_{\vvn,\vvs,\vvt}({\vv\theta}/{\vvn'})}{|\vvn|^2} \prodd k12 \frac{g\pp{(n^*rw_k)^{-1/\alpha_k}, {\theta_k}/{n'_k}}}{(n^*)^{1/\alpha_k}}d\vv\theta\frac{dr}{r^{2}} \Lambda(d\vvw).
\end{multline}
With our choice of $\vvn'$, we have that formally, by \eqref{eq:limit_D} and \eqref{eq:limit_g},
%\[D_{\vvn,\vvs,\vvt}(\vv\theta/\vvn') \sim |\vvn|^2 \frac{(e^{is_1\theta_1}-1)\wb{({e^{it_1\theta_1}-1})}}{|\theta_1|^2}s_2t_2,\]
%and
%\[g\pp{(n^*rw_k)^{-1/\alpha_k}, {\theta_k}/{n'_k}}\sim (n^*)^{p(\vv\alpha)}\prodd k12\frac{2(rw_k)^{-1/\alpha_k}}{(rw_k)^{-2/\alpha_k}+\theta_k^2}.\]
 \eqref{eq:triple} is asymptotically equivalent to
\begin{multline*}
\frac{s_2t_2}{(2\pi)^2\log n_1}\\
\times \int_{\Delta_1}\int_{1/(n^*(w_1\wedge w_2))}^\infty\int_{\R^2}\frac{(e^{is_1\theta_1}-1)\wb{({e^{it_1\theta_1}-1})}}{|\theta_1|^2} \prodd k12\frac{2(rw_k)^{-1/\alpha_k}}{(rw_k)^{-2/\alpha_k}+\theta_k^2}d\vv\theta\frac{dr}{r^{2}} \Lambda(d\vvw).
\end{multline*}
For the above triple integral, we first integrate with respect to $d\theta_2$. The above then  becomes
\[
\frac{s_2t_2}{2\pi \log n_1}\int_{\Delta_1}\int_{\R}\frac{(e^{is_1\theta_1}-1)\wb{(e^{it_1\theta_1}-1)}}{|\theta_1|^2}\int_{1/(n^*(w_1\wedge w_2))}^\infty \frac{2(rw_1)^{-1}}{(rw_1)^{-2}+\theta_1^2}\frac{dr}{r^{2}}d\theta_1 \Lambda(d\vvw).
\]
Now, for the inner integral with respect to $dr$, we write it as
\begin{align*}
\int_0^{n^*(w_1\wedge w_2)}\frac{2rw_1\inv}{r^2w_1^{-2}+\theta_1^2}dr & = w_1\int_0^{n^*(w_1\wedge w_2)/(w_1\theta_1)}\frac{2r}{r^2+1}dr \\
& = w_1\log \bb{1+\pp{n^*\frac{w_1\wedge w_2}{w_1\theta_1}}^2}\sim 2w_1 \log n^*.
\end{align*}
To sum up, we have formally shown that
\begin{align*}
\frac{ \cov(S_{\vvn}(\vvs),S_{\vvn}(\vvt))}{n_1\log (n_1)n_2^2} & =\frac{s_2t_2}\pi\int_{\Delta_1}\int_{-\infty}^{\infty}\frac{(e^{is_1\theta_1}-1)\wb{({e^{it_1\theta_1}-1})}}{|\theta_1|^2}d\theta_1 w_1\Lambda(d\vvw)\\
& = \cov(\mathbb B_{\vvs}^{\vvH}, \mathbb B_{\vvt}^{\vvH})\cdot 4\pi\int_{\Delta_1}w_1\Lambda(d\vvw),
\end{align*}
where $\vvH=(1/2,1)$.
Now we provide a rigorous proof of this convergence. We shall apply several approximations in a row.
Again we start with \eqref{eq:triple}.
Recall our notations for $\Omega_\vvn^{0,0}$ and $\Omega_\vvn$. We first remark that for this integration over
$\Omega_\vvn$
and without the $\log n_1$ normalization, it is same as
\eqref{integral1} with $\Omega_\vvn^{0,0}$ replaced by $\Omega_\vvn$, and can be bounded by $C_\delta$ (see \eqref{eq:J2}, \eqref{eq:J3} and \eqref{eq:J4}).
Therefore, the contribution to the limit comes from the integration over $\Omega_\vvn^{0,0}$.

Set
\[
D_{n,s,t}(\theta) = {\rm Re}\pp{D_{\floor{ns}}(\theta)\wb{D_{\floor{nt}}(\theta)}}.
\]
The imaginary part of the integrand integrates to zero, because of the symmetry.
 We first look at
\[
\psi_\vvn(r,\vvw):= \int_{-n_2'\pi}^{n_2'\pi}\frac{D_{n_2,s_2,t_2}(\theta_2/n_2')}{n_2^2}\frac{g\pp{(n^*rw_2)^{-1/\alpha_2},\theta_2/n_2'}}{(n^*)^{1/2}}
\ind_{\Omega_\vvn^{0,0}}
d\theta_2.
\]
One can show by finding upper and lower bounds of $g$ depending on $\delta$ but not $\theta$,
\equh\label{eq:n2}
\lim_{\delta\downarrow0}\limsup_{\vvn\to\infty}\sup_{n^*r\vvw \in T_{\vv\alpha}((0,\delta])}\abs{\psi_\vvn(r,\vvw) - s_2t_2\pi} = 0.
\eque
From now on, write
\[
\rho = w_1\delta\wedge w_2\delta^{\alpha_2}.
\]
Introduce
\equh\label{eq:n1}
\Psi_{n_1}^\pm(\delta):=\frac{1}{\log n_1}\int_{\Delta_1}\int_{-n_1\pi}^{n_1\pi}\int^\infty_{(n_1\rho)\inv} \frac{D^\pm_{n_1,s_1,t_1}(\theta_1/n_1)}{n_1^2}
\frac{g\pp{(n_1rw_1)^{-1},\theta_1/n_1}}{n_1}\frac{dr}{r^2}d\theta_1\Lambda(dw_1).
\eque
We shall treat the positive and negative parts of $D$, denoted by $D^\pm$, separately, and we accordingly let $\mathfrak D_{s_1,t_1}^\pm(\theta_1)$ denote the positive/negative part of the real part of $(e^{is_1\theta_1}-1)\wb{(e^{it_1\theta_1}-1)}/|\theta_1|^2$.
We shall show
\begin{multline}\label{eq:sandwich}
(2-\delta)\sigma_\pm^2 \le \liminf_{n_1\to\infty}\Psi^\pm_{n_1}(\delta)\le \limsup_{n_1\to\infty}\Psi^\pm_{n_1}(\delta) \le 2\sigma_\pm^2\\
\qmwith \sigma_\pm^2 :=\int_{\Delta_1} w_1\Lambda(d\vvw)\int_\R\mathfrak D^\pm_{s_1,t_1}(\theta_1)d\theta_1  \mbox{ and for all }\delta\in(0,1),
\end{multline}
from which, and thanks to \eqref{eq:n2}, the desired result follows. We only prove the above for $\Psi_{n_1}^+(\delta)$, as the proof for $\Psi_{n_1}^-(\delta)$ is identical.

To simplify further the notation, introduce
\[
\wt g_{n_1}(r,\theta_1,w_1) := \frac{rw_1\inv(2-r(n_1w_1)\inv)}{r^2w_1^{-2}+2(1-r(n_1w_1)\inv)n_1^2(1-\cos(\theta_1/n_1))}.
\]
By change of variable $r\to r^{-1}$, we rewrite \eqref{eq:n1} as
\begin{align*}
\Psi^+_{n_1}(\delta) & = \frac1{\log n_1}\int_{\Delta_1}\int_{-n_1\pi}^{n_1\pi}\int_0^{n_1\rho}\frac{D^+_{n_1,s_1,t_1}(\theta_1/n_1)}{n_1^2}\wt g_{n_1}(r,\theta_1,w_1)drd\theta_1\Lambda(d\vvw)\\
& =:\frac1{\log n_1}\int_{\Delta_1}\int_\R\int_0^{n_1\rho}f_{n_1}(r,\theta_1,w_1)\inddd{\theta_1\le n_1\pi}drd\theta_1\Lambda(d\vvw),
\end{align*}
break the triple integral above into two parts
\equh\label{eq:two_parts}
\int_{\Delta_1}\int_{-n_1\pi}^{n_1\pi}\int_1^{n_1\rho}f_{n_1}drd\theta_1\Lambda(d\vvw) \qmand
\int_{\Delta_1}\int_{-n_1\pi}^{n_1\pi}\int_0^1f_{n_1}drd\theta_1\Lambda(d\vvw),
\eque
and treat them separately. Asymptotically, the former contributes and the latter is negligible.

For the first integral in \eqref{eq:two_parts}, by
\[
\wt g_{n_1}(r,\theta_1,w_1)\inddd{r\in(0,n_1\rho]}\le \frac{2w_1}r,
\]
we have
\begin{multline*}
\frac1{\log n_1}\int_{\Delta_1}\int_\R\int_1^{n_1\rho}f_{n_1}(r,\theta_1,w_1)\inddd{|\theta_1|\le n_1\pi}drd\theta_1\Lambda(d\vvw)\\
\le
2\int_{\Delta_1}w_1\Lambda(d\vvw)\int_\R \frac{D_{n_1}^+(\theta_1/n_1)}{n_1^2}\inddd{|\theta_1|\le n_1\pi}d\theta_1.
\end{multline*}
Therefore,
\equh\label{eq:limsup}
\limsup_{n_1\to\infty}\frac1{\log n_1}\int_{\Delta_1}\int_{-n_1\pi}^{n_1\pi}\int_1^{n_1\rho}f_{n_1}(r,\theta_1,w_1)drd\theta_1\Lambda(d\vvw)\le 2\sigma_+^2.
\eque

For a lower bound, remark that
\[
\wt g_{n_1}(r,\theta_1,w_1) \ge \frac{(2-\delta)w_1r}{r^2+C\theta_1^2w_1^2} \mfa r\in(0,n_1\rho].
\]
We have, for $n_1\rho\ge 1$,
\begin{multline}
 \int_1^{n_1\rho}   \wt g_{n_1}(r,\theta_1,w_1)dr
 \ge \frac{2-\delta}2w_1 \bb{\log \pp{n_1\rho}^2 - \log (1+C\theta_1^2w_1^2)}\\
 \ge (2-\delta)w_1\pp{\log n_1 - C(1+|\log w_1|+|\log w_2|+|\log\delta| + |\log |\theta_1||},\label{eq:lower_bound}
\end{multline}
where in the last step we used inequalities $\log(n_1\rho)\ge \log n_1 - |\log\rho| \ge \log n_1 - 2\pp{|\log w_1|+|\log w_2| + |\log \delta|}$
and $\log(1+C\theta_1^2w_1^2) \le \log 2+|\log (2C\theta_1^2w_1^2)| \le C(1+|\log|\theta_1||+|\log w_1|)$.

Since \eqref{eq:lower_bound} does not hold uniformly in $n_1$ for all $\vvw\in\Delta_1$,  we cannot apply it directly when integrating with respect to $\Lambda(d\vvw)$. Instead, we first apply  the above to the integral restricted to $\{\vvw\in\Delta_1, w_1\in[\epsilon,1-\epsilon]\}$, and then let $\epsilon\downarrow 0$. Eventually we obtain
\equh\label{eq:liminf}
\liminf_{n_1\to\infty}\frac1{\log n_1}\int_{\Delta_1}\int_{-n_1\pi}^{n_1\pi}\int_1^{n_1\rho}f_{n_1}(r,\theta_1,w_1)drd\theta_1\Lambda(d\vvw)\ge (2-\delta)\sigma_+^2.
\eque
In this step, we need the assumption \eqref{eq:Lambda_boundary}.

Next, to deal with the second integral in \eqref{eq:two_parts}, we have, using the upper bound on $D$,
\begin{multline*}
\int_{-n_1\pi}^{n_1\pi}\int_0^1 f_{n_1}(r,\theta_1,w_1) dr d\theta_1 \\
\le C\int_0^1\int_{-n_1\pi}^{n_1\pi}\min\ccbb{1,\frac1{|\theta_1|^2}}\frac{rw_1}{r^2+Cn_1^2(1-\cos(\theta_1/n_1))w_1^2}d\theta_1dr.
\end{multline*}
By restricting the inner integration to $\{\theta_1\in\R: |\theta_1|\le n_1\delta\}$ and $\{\theta_1\in\R: n_1\delta < |\theta_1|\le n_1\pi\}$, and bounding $n_1^2(1-\cos(\theta_1/n_1))$ from below by $C_\delta\theta_1^2$ and $C_\delta$, respectively,  one eventually has
\equh\label{eq:second}
\int_{\Delta_1}\int_{-n_1\pi}^{n_1\pi}\int_0^1f_{n_1}drd\theta_1\Lambda(d\vvw)
\le C_\delta (1+1/n_1) \int_{\Delta_1}w_1\Lambda(d\vvw).
\eque
Combining \eqref{eq:limsup}, \eqref{eq:liminf} and \eqref{eq:second}, we have proved \eqref{eq:sandwich} and hence the desired result.
\end{proof}

\end{document}